\documentclass{article}
\usepackage{tikz}
\usetikzlibrary{arrows.meta,positioning,fit,calc}
\usepackage[a4paper]{geometry}
\usepackage{algorithm,algpseudocode}
\usepackage{graphicx}
\usepackage{amsmath}
\allowdisplaybreaks
\usepackage{amsthm}
\usepackage{amssymb}
\usepackage{hyperref}
\usepackage{subcaption}
\usepackage[nameinlink,capitalise,noabbrev]{cleveref} 
\usepackage{caption}
\usepackage{booktabs}
\usepackage{xcolor}
\usepackage[T1]{fontenc}
\usepackage{multirow}
\usepackage{nicematrix}
\usepackage{cite}
\usepackage{tcolorbox}

\usepackage{bm}
\usepackage{setspace}
\usepackage{enumitem}
\usepackage{mathrsfs}
\usepackage{mathtools}

\newcommand{\dt}{\, \mathrm{d} t}
\newcommand{\dx}{\, \mathrm{d} x}

\newcommand{\e}{\mathrm{e}}

\newcommand{\Dn}[2]{\frac{\mathrm{d}^{#1}}{\mathrm{d}{#2}^{#1}}}

\theoremstyle{plain}
\newtheorem{theorem}{Theorem}[section]
\newtheorem{lemma}[theorem]{Lemma}

\theoremstyle{definition}
\newtheorem{definition}{Definition}

\theoremstyle{remark}
\newtheorem{remark}{Remark}

\definecolor{bkgndcolor}{rgb}{1,1,1}

\newenvironment{keywords}{\medskip\textbf{Keywords:}}{}

\makeatletter
\renewcommand{\theHALG@line}{\thealgorithm.\arabic{ALG@line}}
\makeatother

\title{On the representation for stochastic graph delay propagation}
\author{Shibo Zeng, Weiguo Gao, Yingzhou Li\thanks{Corresponding author.}\\
School of Mathematical Sciences, Fudan University, Shanghai, China\\
\texttt{23110180055@m.fudan.edu.cn}, \texttt{wggao@fudan.edu.cn}, \texttt{yingzhouli@fudan.edu.cn}}
\date{}

\begin{document}
\pagecolor{bkgndcolor}

\maketitle

\begin{abstract}
In this work, we utilize the Hermite expansion to approximate the distributions of the sum and maximum of independent random variables. We model distributions with a three-segment representation, where the left and right tails are respectively modeled as combinations of Hermite functions, and the intermediate segment is approximated by piecewise polynomials.
This approximation admits rigorous $L^2$- and pointwise convergence properties supported by classical results. We develop an algorithmic framework for applying our model to the graph delay propagation problem, where sum and max operations are performed on the proposed model structure. Numerical experiments demonstrate that our model can capture the quantile values with high accuracy compared to Monte Carlo simulation results, significantly outperforming classical Gaussian-based models.

\end{abstract}

\section{Introduction}
\subsection{Background}
This paper studies the problem of estimating delay propagation on a directed acyclic graph (DAG), which has long been a topic of interest in static timing analysis (STA).
\@ Suppose $G = (V,E)$ is a DAG, where $V$ denotes the node set and $E$ denotes the edge set. 
Each node and edge is associated with a \textit{delay}, representing the time it takes for a signal to pass through. 
The delay propagation problem aims to determine the worst-case signal arrival times from the source node(s) to the sink node(s). These arrival times are then checked against the clocking constraints to identify potential timing violations.
These violations may cause registers to capture incorrect data, potentially resulting in function failure of the entire circuit.

There are two general approaches to the delay propagation problem in STA:\@ \textit{graph-based analysis (GBA)} and \textit{path-based analysis (PBA)}. A detailed introduction to these two methods can be found in~\cite{chadha2009sta}. GBA is often used in the first stage of timing analysis, where delays are propagated iteratively along the graph. For each node $v$, the arrival time $D(v)$ is computed in topological order according to  
\begin{equation}\label{eq:delayProp}
    D(v) = \max\limits_{\substack{u\in V\\(u,v)\in E }}\bigl\{ D(u) + d{(u,v)} \bigr\} + d(v),
\end{equation}
where $d(u,v)$ denotes the incremental delay on edge $(u,v)$, and $d(v)$ denotes the incremental delay incurred when the signal passes through node $v$. 

After completing the GBA stage, paths that potentially violate the timing constraints will be identified. The PBA method is then applied to perform a more detailed and accurate examination, though at a higher computational cost. In this paper, we focus on the GBA framework, within which statistical static timing analysis (SSTA) is conducted. 

In SSTA, each delay is modeled as a random variable due to on-chip variation (OCV) arising from spatial and temporal fluctuations, or a sum of random variables, each capturing the uncertainty contributed by a different source of variation. Under this formulation, our objective is to efficiently compute, for every node $v$ in the graph, the accumulated delay (i.e., the arrival time) from the source node(s), based on Eq.~\eqref{eq:delayProp}.

Modeling node delays and edge delays as random variables reduces the computation of $D(v)$ to a sequence of binary operations over random variables, where each operation is either a max operation or a sum operation. To be precise, we consider the following basic operations on independent random variables $X_1$ and $X_2$: 
\[
    X = X_1 + X_2, \quad \text{or} \quad X = \max\{X_1, X_2\}. 
\]
Other binary operations that may arise in practice, such as subtraction and minimum,
\[
    X = X_1 - X_2, \quad \text{or} \quad X = \min\{X_1, X_2\},
\]
can be equivalently transformed into the above forms.
Therefore, without loss of generality, we restrict our discussion to the sum and max operations.

Although $X_1$ and $X_2$ might be statistically dependent in practical circuits, we ignore the correlation between $X_1$ and $X_2$ throughout the calculation, as is commonly assumed in STA~\cite{bosak_statistical_2024}. These two binary operations put forward two crucial problems. One is how to appropriately model the random variables, and the other is how to efficiently and accurately evaluate the max and sum operations.

A widely adopted simplification is to model all delay random variables as Gaussian. The probabilistic density function (PDF) of a Gaussian random variable $\mathcal{N}(\mu,\sigma^2)$ is given by 
\begin{equation}
    \phi(x;\mu,\sigma) \coloneqq \frac{1}{\sqrt{2\pi}\sigma}\e^{-\frac{(x-\mu)^2}{2\sigma^2}}.    
\end{equation}
When $\mu = 0$ and $\sigma = 1$, this expression reduces to the PDF of the standard normal distribution $\mathcal{N}(0, 1)$, which we denote by $\phi(x)$ for brevity.

The cumulative distribution function (CDF) of a Gaussian variable $\mathcal{N}(\mu,\sigma^2)$ is given by 
\begin{equation}
    \Phi(x;\mu,\sigma) \coloneqq \Phi\left(\frac{x-\mu}{\sigma}\right) = \frac{1}{2}\left[1 + \operatorname{erf}\left(\frac{x-\mu}{\sqrt{2}\sigma}\right)\right],
\end{equation}
where $\operatorname{erf}(x) = \frac{2}{\sqrt{\pi}} \int_{0}^{x}\e^{-t^2}\dt$ is the standard error function.
However, Gaussian random variables fail to capture the characteristics of delay, since the advancement of manufacturing technology has made the distribution of gate delay and interconnect delay increasingly complex. Moreover, the maximum of independent Gaussian random variables is not necessarily Gaussian. When more factors are taken into account, such as the skewness of the random variables, more sophisticated models become necessary. 

\subsection{Related Work}
There have been massive studies concerning the delay propagation problem. The earliest methods, as mentioned above, modeled each node delay and interconnect delay as a Gaussian random variable. Their algorithms were based on the idea of the classic Clark's algorithm~\cite{ClarkGreatest1961}. For example, Berkelaar~\cite{berkelaar1997statistical} proposed a linear-time algorithm for Gaussian delay propagation under the assumption that all delays are independent, while Tsukiyama et al.~\cite{TsukiyamaCorrelation2001} considered the correlations between delays. 

A more advanced Gaussian framework was proposed by Kang et al.~\cite{Kunhyuk_Statistical_2005}, who introduced Levelized Covariance Propagation (LCP), in which covariance information is maintained and propagated level-by-level. Using Clark's algorithm, LCP computes the mean $\mu_{\max}$ and standard deviation $\sigma_{\max}$ of the maximum of correlated Gaussian random variables and then projects the result back to $\mathcal{N}(\mu_{\max}, \sigma_{\max}^2)$ for further propagation. This can be viewed as a \textit{moment-matching} approach, which is also useful for sums of distributions that are not closed under addition. 

Chopra et al.~\cite{Chopra_A_2006} took skewness into consideration by modeling the distributions to propagate as skew-normal random variables. They proposed a statistical max operation method that captures the skewness of the delay distributions, which shows improved accuracy results compared to the classical Clark's approach. 

Hcine and Bouallegue~\cite{ben_hcine_fitting_2014} made a similar effort, where the sum of independent lognormal random variables was approximated by a log-skew-normal distribution. However, instead of using the moment-matching method as proposed by Fenton~\cite{FentonLN1960}, the authors chose a tail-matching method to capture the tail asymptotics of the desired distribution. These log-type distributions are suitable for modeling delay since they naturally preserve positivity during propagation. 

Recently, Jin et al.~\cite{JinKurtosis2022} proposed a statistical delay model based on the log-extended-skew-normal (LESN) distribution, which extends the conventional log-skew-normal model by incorporating kurtosis control. They also applied the moment-matching method. In particular, the LESN model has a precise prediction of the $3\sigma$ delay, which is crucial for the worst-case analysis.

A different model has been formulated by Cheng et al.~\cite{ChengPolynomial2008} where all delays to propagate were modeled as quadratic polynomials of several independent random variables. While this model supports closed-form delay summation, it also provides an approximation method for computing the maximum of independent random variables by second-degree polynomial fitting. 

Apart from the method using moment-matching and polynomial fitting, Azuma et al.~\cite{AzumaApproximating2017} introduced a novel framework for approximating the maximum of Gaussian random variables by a Gaussian mixture model (GMM). Utilizing the fact that the convolution of PDFs of Gaussian random variables is still a Gaussian PDF, computing the sum in this model is straightforward. To address the maximum problem, the authors also used a moment-matching method to fit the first two moments of the maximum of GMMs. Lately, a series of works~\cite{FreeleyStatistical2018, mishagli_radial_2020, mishagli_gate_2024} have been proposed to further apply this method to the large-scale graph propagation problem. Furthermore, Bos{\'a}k et al.~\cite{bosak_statistical_2024} used a histogram-based approach and optimization methods to propagate delay distributions on a graph. 

Our work is inspired not only by the works mentioned above, but also by the classical Hermite polynomial expansion theory. The Gram--Charlier series, for example, is a well-known method for approximating a probability distribution by a series of Hermite polynomials (for probabilists) multiplied by a Gaussian kernel through moment-matching, as discussed in~\cite{cramer1999mathematical}. In recent years, there have also been several works that apply this method to problems in different fields. For example, Capodaglio et al.~\cite{capodaglio_approximation_2021} applied the Gram--Charlier series to approximate the PDF of solutions to PDEs with random parameters, and Dufresne and Li~\cite{DufresneLi2016} used it in Asian option pricing. However, all these methods face the poor pointwise convergence property of the Gram--Charlier series, as was discussed in~\cite{blinnikov_expansions_1998}. 

\subsection{Contributions}
We present a general framework for representing and propagating random variables on DAGs under binary operations such as max and sum. Each distribution is approximated by a three-segment formulation consisting of polynomial-Gaussian tail models and an interpolation approach across the intermediate segment. During the propagation process, this structure is preserved, whose convergence properties are proved base on the theory of the Hermite function expansion. Moreover, we presented an algorithmic framework to evaluate the binary operations and reproject the input and resulting distributions back into our model family through CDF sampling and a nonlinear least-squares tail-fitting procedure. This yields a complete and operational representation that remains stable under propagation. Numerical experiments are conducted to demonstrate that our model and algorithm can capture the quantile values with high accuracy with respect to Monte Carlo simulation results, which improves over classical Gaussian-based models.

\medskip
The rest of this paper is organized as follows.
In Sec.~\ref{sec:pre}, we introduce briefly the mathematical preliminaries that will be used later in our model.
In Sec.~\ref{sec:model}, we present the mathematical formulation of the delay propagation model.
In Sec.~\ref{sec:complete}, we illustrate the completeness of the model by proving the $L^2$-convergence and pointwise convergence of our model.
In Sec.~\ref{sec:imp}, we introduce algorithms for propagating the proposed distributions through sum and max operations.
In Sec.~\ref{sec:numeric}, we present numerical results of our model and our proposed algorithms.

\section{Preliminaries}\label{sec:pre}
In this section, we summarize the mathematical preliminaries needed for the analysis in subsequent sections. To support the convergence results established later, we review the standard and parametric Hermite polynomials and derive their associated orthogonality relations.
\subsection{Standard Hermite polynomials}
Unless otherwise stated, we adopt the physicists' convention for the Hermite polynomials, defined by
\[
H_k(x) \coloneqq (-1)^k\e^{x^2}\Dn{k}{x}\e^{-x^2}, \quad k = 0, 1, 2, \dotsc .
\]
We refer to these as the \textit{standard} Hermite polynomials to distinguish them from the \textit{parametric} Hermite polynomials, which will be discussed in the following subsection. 

Function $H_k(x)$ is a polynomial of degree $k$ in $x$. The Hermite polynomials satisfy the following three-term recurrence relation:
\begin{equation*}\label{eq:HermiteRecurrence}
    H_{k+1}(x) = 2xH_k(x) - 2kH_{k-1}(x), \quad k = 1, 2, \dotsc ,
\end{equation*}
with $H_0(x) = 1$ and $H_1(x) = 2x$.
They are orthogonal with respect to the weight function $\omega(x)\coloneqq \e^{-x^2}$: 
\[
\int_{-\infty}^{+\infty} H_m(x)H_n(x)\e^{-x^2}\dx = \gamma_n\delta_{mn},
\]
where the normalization constant $\gamma_n$ is given by $\gamma_n \coloneqq 2^n n!\sqrt{\pi}$.

\subsection{Parametric Hermite polynomials and Hermite functions}
For better approximation capability, we use the parametric Hermite polynomials/functions, where we modify the weight function as $\omega(x;\mu,\sigma) \coloneqq \e^{-{(x-\mu)^2}/{\sigma^2}}$ for $\sigma > 0$ and $\mu$ being real parameters. 
Also, we denote by $H_k(x;\mu,\sigma) \coloneqq H_k(\frac{x-\mu}{\sigma})$ the parametric Hermite polynomials.

The orthogonality of the parametric Hermite polynomials can be established through a straightforward computation:
\[
\int_{-\infty}^{+\infty} H_m(x;\mu,\sigma)H_n(x;\mu,\sigma)\e^{-\frac{(x-\mu)^2}{\sigma^2}}\dx = \sigma\gamma_n\delta_{mn}.
\]
The parametric Hermite functions $\psi_k(x;\mu,\sigma)$ are defined by
\[\psi_k(x;\mu,\sigma) \coloneqq \frac{1}{\sqrt{\sigma\gamma_k}} H_k(x;\mu,\sigma)\e^{-\frac{(x-\mu)^2}{2\sigma^2}}.\] 
The $L^2$-completeness of \(\{\psi_k(x;\mu,\sigma)\}_{k=0}^{\infty}\) is guaranteed by the following lemma, which is a direct generalization of the standard case in~\cite{szeg1939orthogonal}.

\begin{lemma}[$L^2$-convergence of the parametric Hermite function expansion]\label{lem:HermtitepolyCompletenessL2param}
    The set of parametric Hermite functions \( \{\psi_k(x;\mu,\sigma)\}_{k=0}^{\infty} \) forms a complete orthogonal basis in \( L^2(\mathbb{R}) \).
    That is, for any function \( f \in L^2(\mathbb{R}) \), there exists a unique sequence of coefficients \( \{c_k(\mu,\sigma)\}_{k=0}^{\infty} \) such that we can write $f(x)$ as 
    \begin{equation*}
     f(x) = \sum_{k=0}^{\infty} c_k(\mu,\sigma) \psi_k(x;\mu,\sigma).   
    \end{equation*}
    The coefficients \( c_k(\mu,\sigma) \) are given by the inner product in \(  L^2(\mathbb{R})  \):
    \[
    c_k(\mu,\sigma) = \int_{-\infty}^{+\infty} f(x) \psi_k(x;\mu,\sigma) \dx.
    \]
    The projection of \( f(x) \) onto the span of the first \( N \) parametric Hermite functions is given by
    \[
    \mathscr{P}_N^f(x;\mu,\sigma) = \sum_{k=0}^{N} c_k(\mu,\sigma) \psi_k(x;\mu,\sigma).
    \]
    Then, the sequence of \( \mathscr{P}_N^f(x) \) converges to \( f(x) \) in the \( L^2 \)-norm:
    \[
    \lim_{N \to \infty} \int_{-\infty}^{\infty} | f(x) -  \mathscr{P}_N^f(x;\mu,\sigma)|^2  \dx = 0.
    \]
\end{lemma}

The pointwise convergence property of the parametric Hermite polynomial expansion is
analogous to that of the standard case given in~\cite{uspensky_development_1926}.
Indeed, the parametric result is obtained from the standard case by the affine transform
\(x \mapsto (x-\mu)/\sigma\). Moreover, the pointwise convergence of the parametric
Hermite \emph{function} expansion for \(f \in L^2(\mathbb{R})\) reduces to the pointwise
convergence of the parametric Hermite \emph{polynomial} expansion for
\(g=f\,\sqrt{\omega(\,\cdot\,;\mu,\sigma)} \in L^2_{\omega(\,\cdot\,;\mu,\sigma)}(\mathbb{R})\).

We state the pointwise convergence property of the parametric Hermite function expansion in the lemma below, paralleling the main theorem of~\cite{uspensky_development_1926}. 

\begin{lemma}[Pointwise convergence of the parametric Hermite function expansion]\label{lem:HermtitefuncCompletenessPointwise}
    Suppose \( f(x) \) satisfies the following conditions:
    \begin{enumerate}[label=(\arabic*)]
        \item \(f(x)\) is absolutely integrable in any finite interval,
        \item \(f(x)\) is of limited variation in a certain interval \((x-\delta,x+\delta)\) for some \(\delta > 0\),
        \item \(f(x)\) belongs to \(L^2(\mathbb{R})\).
    \end{enumerate}
    Then the series expansion $\sum\limits_{k=0}^{\infty} c_k(\mu,\sigma) \psi_k(x;\mu,\sigma)$ converges and satisfies
    \[
    \sum\limits_{k=0}^{\infty} c_k(\mu,\sigma) \psi_k(x;\mu,\sigma) = \frac{f(x^{+}) + f(x^{-})}{2}.
    \]
\end{lemma}

\section{Model formulation}\label{sec:model} 
To accurately capture the delay distributions during propagation, especially in the tail segments, we observed that Gaussian models tend to underestimate the probability of large deviations. Inspired by the observation that empirical delay distributions often exhibit longer or skewed tails, we adopt a polynomial-modulated Gaussian model, where the Gaussian kernel is multiplied by a low-degree polynomial. This allows more flexibility in shaping the PDF while maintaining a Gaussian-like functional form. 

To further support efficient computation, we represent the intermediate segment PDF with piecewise polynomials, yielding a three-segment model for the PDF on the whole real line. The resulting form retains analytical tractability while significantly enhancing the approximation capacity.

\subsection{Basic PDF model}\label{sub-sec:PDFmodel}
We adopt a three-segment PDF model for propagating independent random variables on a graph.
The PDF of a random variable $X$ to be propagated along the graph is modeled as  
\begin{equation} \label{eq:modelPDF}
f(t) = \left\{
\begin{aligned}
    &\sum\limits_{j=0}^N c_j^{(l)}t^j\phi\bigl(t;\mu^{(l)},\sigma^{(l)}\bigr), && t \leq q_0, \\
    &P_i(t), && q_{i-1} \leq t < q_i, \; i = 1,2, \dotsc, n, \\
    &\sum\limits_{j=0}^N c_j^{(r)}t^j\phi\bigl(t;\mu^{(r)},\sigma^{(r)}\bigr), && t \geq q_n.
\end{aligned}
\right.
\end{equation}
Here $n$ denotes the number of partition points in the intermediate segment, and $N$ is the degree of the Gaussian-tail polynomials, both user-specified. The PDF on each subinterval $[q_{i-1}, q_i]$ in the intermediate segment is modeled by a polynomial $P_i(t)$, which is constructed to ensure continuity and smoothness of the overall PDF.\@
Unless otherwise stated, we use an equally spaced partition in the intermediate segment, i.e., $q_i-q_{i-1} = (q_n - q_0)/n$, $i = 1,2, \dotsc, n$. 

Suppose $F(x)$ is the CDF of $X$. We choose $q_0$ and $q_n$ as the quantile values corresponding to two prescribed CDF levels of $X$ for each random variable $X$ on the entire graph. For example, taking the CDF levels $0.135\%$ and $99.865\%$ yields the $\pm 3\sigma$ quantiles in the Gaussian case.
\begin{remark}
    (a) In our general model~\eqref{eq:modelPDF}, the intermediate segment of the PDF are piecewise polynomials of arbitrary order to preserve the regularity of the PDF.\@ However, in practice we typically assume the intermediate segment to be piecewise constant, i.e., $P_i(t) = (F(q_i) - F(q_{i-1}))/(q_i - q_{i-1})$, as this choice tends to maintain the nonnegativity of the PDF.\@ Selecting the degree of the intermediate polynomials involves a trade-off between model complexity and approximation accuracy. A more detailed study of the convergence rate for higher-degree models is deferred to future work. For the sake of simplicity, we assume $P_i(t)$ to be piecewise constant throughout the remainder of this paper. 

    (b) Our PDF model does not inherently provide non-negativity, which is a fundamental requirement for any valid probability density function. To address this issue, one may impose constraints on the leading coefficients of the tail polynomials. In the subsequent analysis of convergence properties, we restrict our attention to the case where $f$ is indeed a valid probability density function.
\end{remark}

\subsection{Basic CDF model}\label{sub-sec:CDFmodel}
During delay propagation, our primary interest lies in the CDF of the propagated random variables, since its behavior determines the target quantiles used to approximate the worst-case delay.

Given a random variable $X$ whose PDF takes the form~\eqref{eq:modelPDF}, the CDF of $X$ can be evaluated numerically at any point $x$. When $ q_0\leq x \leq q_n$, the CDF can be evaluated efficiently due to the specific structure of the intermediate segment. For $x < q_0$ or $x > q_n$, the CDF can be respectively determined by a linear combination of 
\[
I_i(x;\mu,\sigma) \triangleq \int_{-\infty}^{x} t^i \phi(t;\mu,\sigma)\,\dt,
\quad \text{or} \quad
J_i(x;\mu,\sigma) \triangleq \int_{x}^{\infty} t^i \phi(t;\mu,\sigma)\,\dt.
\]
We first derive a recurrence relation for $I_i(x;\mu,\sigma)$ by induction. Based on the relation 
\[
\frac{\mathrm{d}}{\mathrm{d}t}\phi(t;\mu,\sigma) = -\frac{t-\mu}{\sigma^2} \phi(t;\mu,\sigma),
\]
we obtain for arbitrary $i \geq 2$
\begin{align*}
    I_i(x;\mu,\sigma) &= \int_{-\infty}^x t^i\phi(t;\mu,\sigma)\dt \\
    &= \mu I_{i-1}(x;\mu,\sigma) -\sigma^2 \int_{-\infty}^{x} t^{i-1}\frac{\mathrm{d}}{\mathrm{d}t}\phi(t;\mu,\sigma) \dt\\
    &= \mu I_{i-1}(x;\mu,\sigma) - x^{i-1}\sigma^2 \phi(x;\mu,\sigma) + (i-1)\sigma^2 I_{i-2}(x;\mu,\sigma).
\end{align*}

For $i = 1$, a similar derivation yields
\[
I_1(x;\mu,\sigma) = \mu I_0(x;\mu,\sigma) - \sigma^2 \phi(x;\mu,\sigma).
\]
Thus, with a linear combination of elements in $\{I_i(x;\mu,\sigma)\}_{i\leq N}$, we can achieve any element in 
\[
\mathrm{span}\left\{
\Phi(x;\mu,\sigma),\ 
\phi(x;\mu,\sigma),\ 
x\phi(x;\mu,\sigma),\ 
\dots,\ 
x^{N-1}\phi(x;\mu,\sigma)
\right\}.
\]
Also, as shown in the derivation above, the transition matrix between the two sets of functions is upper-triangular with nonzero diagonal entries, and is therefore invertible.

By applying the same argument to $J_i$ for the right tail (see Appendix~\ref{AppendixA}), we obtain the following lemma:
\begin{lemma}\label{lem:isomorphism}
Fix $\mu\in\mathbb{R}$, $\sigma>0$, and $N\ge 1$. Then 
\begin{enumerate}[label=(\arabic*)]
\item The function spaces
\[
\mathrm{span}\bigl\{I_0(x;\mu,\sigma),I_1(x;\mu,\sigma),\dots,I_N(x;\mu,\sigma)\bigr\}
\]
and
\[
\mathrm{span}\bigl\{\Phi(x;\mu,\sigma),\phi(x;\mu,\sigma),x\phi(x;\mu,\sigma),\dots,x^{N-1}\phi(x;\mu,\sigma)\bigr\}
\]
coincide.
\item The function spaces
\[
\mathrm{span}\bigl\{J_0(x;\mu,\sigma),J_1(x;\mu,\sigma),\dots,J_N(x;\mu,\sigma)\bigr\}
\]
and
\[
\mathrm{span}\bigl\{1-\Phi(x;\mu,\sigma),\phi(x;\mu,\sigma),x\phi(x;\mu,\sigma),\dots,x^{N-1}\phi(x;\mu,\sigma)\bigr\}
\]
coincide. 
\end{enumerate} 
\end{lemma}

During delay propagation, we need to compute the sum, subtraction, maximum and minimum of two independent random variables $X_1$ and $X_2$ with known PDFs and CDFs. Since 
\[
X_1 - X_2 = X_1 + (-X_2), \quad\; \text{and} \quad\; \min(X_1, X_2) = -\max(-X_1, -X_2),
\]
the subtraction and minimum operations can be reduced to sum and max operations, respectively. These reductions only involves sign inversion and therefore preserves independence. Thus, we consider only the sum and max operations in the following discussion. Let $F_1$ and $F_2$ denote the CDF of $X_1$ and $X_2$. The CDF of $X_1 + X_2$ and $\max(X_1, X_2)$ can be respectively given by the following formulae:
\begin{align*}
    &F_{X_1 + X_2}(x) = \int_{-\infty}^{+\infty} F_1(x - t) f_2(t) \dt = \int_{-\infty}^{+\infty} f_1(t) F_2(x-t) \dt, \\
    &F_{\max(X_1, X_2)}(x) = F_1(x) F_2(x).
\end{align*}
Using the formulae above, we can numerically evaluate the CDF of 
\(
X\coloneqq X_1\odot X_2
\)
at any point on $\mathbb{R}$, where $\odot$ denotes either the sum or the max operation. 

In the intermediate segment, we can use interpolation methods to derive the piecewise polynomial approximation of the CDF.\@ In the tail segments, we aim to approximate the exact CDF of $X$ by a linear combination of $\{I_i(x;\mu,\sigma)\}_{i=0}^N$ or $\{J_i(x;\mu,\sigma)\}_{i=0}^N$. In particular, we focus on approximating the $L^2$-error and the $L^\infty$-error between the target CDF and the approximation. 
We present the derivation for the left tail below, while the right tail can be treated similarly and is shown in Appendix~\ref{AppendixA}. Also, only the max case is presented, as the sum case can be handled similarly.
The target functions can be respectively written as: 
\[
\int_{-\infty}^{q_0}\Big(F_1(x)F_2(x)-\sum\limits_{i=0}^N c_i^{(l)}I_i(x;\mu^{(l)},\sigma^{(l)})\Big)^2\dx
\qquad \text{or} \qquad 
F_1(x)F_2(x)-\sum\limits_{i=0}^N c_i^{(l)}I_i(x;\mu^{(l)},\sigma^{(l)}), \quad \forall x \leq q_0,
\]
where $\{c_i^{(l)}\}_{i=0}^N, \mu^{(l)}, \sigma^{(l)}$ are real coefficients, and $q_0$ is a left quantile point for $\max(X_1, X_2)$.

Take the left-tail $L^2$-error as an example, where a similar derivation holds for $L^\infty$-error. According to Lem.~\ref{lem:isomorphism}, we can rewrite the left-tail model CDF as 
\begin{equation}\label{eq:leftTailCDF}
    F(x) = \alpha^{(l)} \Phi(x;\mu^{(l)},\sigma^{(l)}) + \sum\limits_{i=0}^{N-1}\beta_i^{(l)} x^i\phi(x;\mu^{(l)},\sigma^{(l)}),
\end{equation}
where $\alpha^{(l)}$, $\{\beta_i^{(l)}\}_{i=0}^{N-1}$, $\mu^{(l)}$, $\sigma^{(l)}$ are the parameters to be decided. Thus, minimizing the $L^2$-error is equivalent to minimizing
\begin{equation}
\int_{-\infty}^{q_0}\Bigg(\Big(F_1(x)F_2(x)-\alpha^{(l)} \Phi(x;\mu^{(l)},\sigma^{(l)}) \Big)-\sum\limits_{i=0}^{N-1}\beta_i^{(l)} x^i\phi(x;\mu^{(l)},\sigma^{(l)})\Bigg)^2\dx.
\label{eq:leftTailL2Error}
\end{equation}
If the Gaussian parameters $\mu^{(l)}$ and $\sigma^{(l)}$ are fixed, the problem of minimizing $L^2$-error can be seen as a linear least-squares problem concerning the coefficients $\alpha^{(l)}$ and $\{\beta_i^{(l)}\}_{i = 0}^{N-1}$. 

\section{Model justification and main results}\label{sec:complete}
In this section, we present $L^2$- and pointwise convergence properties of the proposed model. Throughout the following discussion, we focus on the theoretical results concerning approxmiating $X = X_1\odot X_2$ with $\widetilde{X}$, where $X_k$, $k = 1, 2$ follow PDFs $f_k$ of the form~\eqref{eq:modelPDF}:
\begin{equation}\label{eq:X12PDF}
    f_k(t) = \left\{
        \begin{aligned}
            &\sum\limits_{j=0}^{N_k} c_{k,j}^{(l)}t^j\phi(t;\mu_k^{(l)},\sigma_k^{(l)}), && t \leq q_{k,0}, \\
            &P_{k,i}(t), && q_{k,i-1} \leq t < q_{k,i}, \;  i = 1,2, \dotsc, n_k, \\
            &\sum\limits_{j=0}^{N_k} c_{k,j}^{(r)}t^j\phi(t;\mu_k^{(r)},\sigma_k^{(r)}), && t \geq q_{k,n_k},
        \end{aligned}
        \right.
\end{equation}
and $\widetilde{X}$ also follows a PDF of the form~\eqref{eq:modelPDF}. In what follows, we restrict our attention to the case where the intermediate pieces $P_{k,i}$ are constants on each subinterval.
Before proceeding to justify our model, we introduce the class of random variables $\mathcal{X}$:
\begin{definition}\label{def:randomVariableSpace}
We denote by $\mathcal{X}$ the class of real-valued continuous random variables $X$
with CDF $F$ such that there exist $\varepsilon > 0$, $\delta > 0$, and points $x_1, x_2 \in \mathbb{R}$ satisfying 
\[
F(x_1) = \varepsilon, \qquad F(x_2) = 1 - \varepsilon,
\]
and 
\[
F \in L^2(-\infty,\, x_1 + \delta) \quad \text{and} \quad 1-F \in L^2(x_2 - \delta,\, +\infty).
\]
We refer to $\mathcal{X}$ as the \emph{class of random variables with $L^2$-integrable tails}.
\end{definition}

\subsection{Expressive power of the model}

We first consider any random variable in class $\mathcal{X}$, and prove that it can be approximated by our model. This result characterizes the expressive power of our model: before performing delay propagation on the DAG, every input random variable must be projected onto the model space, and the approximation ensures that this projection is well-defined.

\begin{theorem}[Expressive power of the model]\label{thm:modelExpressiveness}
    Suppose $X\in \mathcal{X}$ is a real-valued continuous random variable with cumulative distribution function $F$. Then there exists $\{f_{N,n}\}$, each element of which follows a PDF of the form~\eqref{eq:modelPDF}, such that
    \begin{align}
        &\lim_{n,N\to\infty}  \Big\Vert F(x) - \int_{-\infty}^{x}f_{N,n}(t)\dt\Big\Vert_2 = 0, \label{eq:L2limit} \\
        &\lim_{n,N\to\infty}  F(x) - \int_{-\infty}^{x}f_{N,n}(t)\dt= 0, \qquad \forall x \in \mathbb{R}. \label{eq:pointwiselimit}
    \end{align}
    Here $q_0$ and $q_n$ in~\eqref{eq:modelPDF} are two fixed parameters, and we take $x_1 = q_0$, $x_2 = q_n$ in Def.~\ref{def:randomVariableSpace}.
\end{theorem}
\begin{proof}
    We first use the Hermite polynomial theory to approximate the left tail and the right tail in a constructive manner. Without loss of generality, we assume that $\mu^{(l)} = 0$, $\sigma^{(l)} = 1$, $\mu^{(r)} = 0$, and $\sigma^{(r)} = 1$.  
    
    By Lem.~\ref{lem:HermtitepolyCompletenessL2param} and Lem.~\ref{lem:HermtitefuncCompletenessPointwise}, together with the preservation of $L^2$-integrability under finite sums, we obtain: 
    \[
    \left\{
    \begin{aligned}
    &F(x) - \alpha^{(l)} \Phi(x) \in L^2(-\infty, q_0+\delta) , &\qquad \forall \alpha^{(l)} \in \mathbb{R}, \\ 
    &1-F(x) - \alpha^{(r)} (1-\Phi(x)) \in L^2(q_n-\delta, +\infty), &\qquad \forall \alpha^{(r)} \in \mathbb{R}.
    \end{aligned}
    \right.
    \]
    Thus, we can approximate $F(x) - \alpha^{(l)}\Phi(x)$ on the left tail and $1-F(x) - \alpha^{(r)}(1-\Phi(x))$ on the right tail using the Hermite function expansion. 

    Consider the left-tail case. Due to the Hermite function expansion convergence results in Lem.~\ref{lem:HermtitepolyCompletenessL2param} and Lem.~\ref{lem:HermtitefuncCompletenessPointwise}, there exist coefficients \(\{\beta_i^{(l)}\}_{i = 0}^{+\infty}\) such that 
    \[
    \left\{
        \begin{aligned}
            &\lim_{N\to+\infty} \sum\limits_{i = 0}^{N-1} \beta_i^{(l)} x^i\phi(x) = F(x) - \alpha^{(l)} \Phi(x), \qquad \forall x \leq q_0, \\
            &\lim_{N\to+\infty} \lVert F(x) - \alpha^{(l)}\Phi(x) - \sum\limits_{i=0}^{N-1}\beta_i^{(l)} x^i \phi(x)\rVert_{L^2\left(-\infty, q_0\right]} = 0.
        \end{aligned}
    \right.
    \]
    We define $F_N^{(l)}(x) \coloneqq \alpha^{(l)}\Phi(x) + \sum\limits_{i=0}^{N-1}\beta_i^{(l)} x^i \phi(x)$ for $x \leq q_0$. Due to the isomorphism in Lem.~\ref{lem:isomorphism} between 
    \[
    \mathrm{span}\left\{ \int_{-\infty}^{x}\phi(t)\dt, \int_{-\infty}^{x}t\phi(t)\dt, \dotsc, \int_{-\infty}^{x}t^N\phi(t)\dt \right\}
    \] 
    and 
    \[
    \mathrm{span}\left\{\Phi(x), \phi(x), \dotsc, x^{N-1}\phi(x)\right\},
    \]
    we can map each $F_N^{(l)}(x)$ back to $f_N^{(l)}(x)$ that takes the left-tail form of~\eqref{eq:modelPDF} and satisfies \(F_N^{(l)}(x) = \int_{-\infty}^x f_N^{(l)}(t)\dt\). We can also get the right-tail expression for $f_N^{(r)}(x)$ in the same manner. 
    
    Thus, by defining 
    \[
        f_N(x) \coloneqq 
        \begin{cases}
            f_N^{(l)}(x), & x \in (-\infty, q_0], \\[6pt]
            f_N^{(r)}(x), & x \in [q_n, +\infty),
        \end{cases}
    \] 
    and 
    \begin{equation}\label{eq:fNx}
        F_N(x) \coloneqq 
        \begin{cases}
            \int_{-\infty}^x f_N^{(l)}(t)\dt, & x \in (-\infty, q_0], \\[6pt]
            1- \int_x^{+\infty} f_N^{(r)}(t)\dt, & x \in [q_n, +\infty),
        \end{cases}    
    \end{equation}
    we obtain $f_N(x)$ and $F_N(x)$ on $(-\infty, q_0] \cup [q_n, +\infty)$ that satisfy
    \[
        \begin{cases}
            \lim\limits_{N\to\infty}  \Big\Vert F(x) - F_N(x)\Big\Vert_{L^2(-\infty, q_0]\cup [q_n, +\infty)} = 0, \\
            \lim\limits_{N\to\infty}  F(x) - F_N(x) = 0, \qquad \forall x \in (-\infty, q_0]\cup [q_n, +\infty).
        \end{cases}    
    \]

    Next, we extend $f_N$ onto $(q_0,q_n)$ by defining $f_{N,n}$ for each $n > 1$ as
    \[
        f_{N,n}(x)=
        \begin{cases}
            f_N(x), & x\in(-\infty,q_0]\cup[q_n,+\infty),\\[6pt]
            \dfrac{F(q_{1})-F_N(q_{0})}{q_{1}-q_{0}}, & x\in(q_0,q_1),\\[10pt]
            \dfrac{F(q_{i+1})-F(q_{i})}{q_{i+1}-q_{i}}, & x\in[q_i,q_{i+1}),\ \ i=1,\dots,n-2,\\[10pt]
            \dfrac{F_N(q_{n})-F(q_{n-1})}{q_{n}-q_{n-1}}, & x\in[q_{n-1},q_n).
        \end{cases}
    \]
    Now we can finally define $F_{N,n}$ as a CDF-form function: 
    \begin{equation}\label{eq:fNnx}
        F_{N,n}(x) \coloneqq \int_{-\infty}^{x} f_{N,n}(t)\dt, \qquad x\in \mathbb{R}.
    \end{equation}
    Here we note that although the representations of Eq.~\eqref{eq:fNx} and Eq.~\eqref{eq:fNnx} on the right tail are different, we still have
    \[
        F_{N,n}(x) = 1 - \int_{x}^{+\infty} f_{N,n}(t)\dt, \qquad x \in [q_n,+\infty),
    \]
    since
    \[
    \int_{-\infty}^{q_n} f_{N,n}(t)\dt + \int_{q_n}^{+\infty} f_{N,n}(t)\dt = F_N(q_n) + \bigl(1-F_N(q_n)\bigr) = 1.
    \]
    Thus, the convergence results of~\eqref{eq:L2limit} and~\eqref{eq:pointwiselimit} restricted to $(-\infty, q_0]\cup [q_n, +\infty)$ are guaranteed, as 
    \[
    F_{N,n}(x) = F_N(x), \quad x\in (-\infty, q_0]\cup [q_n, +\infty).
    \]

    For the intermediate segment, we define the auxiliary functions
    \[
    \begin{aligned}
    h_n(t) &= \frac{F(q_i)- F(q_{i-1})}{q_i - q_{i-1}}, \qquad q_{i-1} \leq t < q_i,\quad i = 1,2, \dotsc, n,\\
    H_n(x) &= F(q_0) + \int_{q_0}^{x} h_n(t)\dt, \qquad x\in [q_0, q_n].
    \end{aligned}
    \]

    For the $L^2$-convergence in the intermediate segment, by the triangular inequality of the $L^2$-norm we have 
    \begin{align*}
        \lVert F_{N,n}- F \rVert_{L^2(q_0,q_n)} \leq \lVert F_{N,n}- H_n \rVert_{L^2(q_0,q_n)} 
        + \lVert H_n - F \rVert_{L^2(q_0,q_n)}.
    \end{align*}
    To estimate $\lVert F_{N,n}- H_n \rVert_{L^2(q_0,q_n)}$, we utilize the fact that 
    $F_{N,n}$ and $H_n$ differ only on the first and the last subintervals. Direct computation yields
    \begin{align*}
    \lVert F_{N,n}- H_n \rVert_{L^2(q_0,q_n)}^2 
    &= \frac{1}{3}(q_1 - q_0)\bigl(F(q_0)-F_{N}(q_0)\bigr)^2 
    + \frac{1}{3}(q_n - q_{n-1})\bigl(F(q_n)-F_{N}(q_n)\bigr)^2 \\
    &= \frac{1}{3n}(q_n - q_0)\bigl(F(q_0)-F_{N}(q_0)\bigr)^2 
    + \frac{1}{3n}(q_n - q_0)\bigl(F(q_n)-F_{N}(q_n)\bigr)^2.
    \end{align*}
    Thus, we have 
    \[
    \lim_{n \to \infty} 
    \lVert F_{N,n}- H_n \rVert_{L^2(q_0,q_n)} = 0.
    \]  

    The term $\lVert H_n - F \rVert_{L^2(q_0,q_n)}$ represents the $L^2$-error of approximating $F$ on $(q_0, q_n)$ by linear interpolation. Given that $F$ is an actual CDF, it is uniformly continuous on $[q_0, q_n]$. Accordingly, for arbitrary $\varepsilon > 0$, we can find $\delta > 0$, such that for all $x, y \in [q_0, q_n]$ satisfying $\lvert x - y\rvert < \delta$, we have $\lvert F(x) - F(y)\rvert < \varepsilon$. Choosing $n$ sufficiently large so that $\max_i (q_i-q_{i-1})<\delta$, we obtain 
    \[
    \lVert H_n - F \rVert_{L^2(q_0,q_n)} = \sum\limits_{i=1}^n \int_{q_{i-1}}^{q_i} \bigl| H_n(x) - F(x)\bigr|^2 \dx < \sum\limits_{i=1}^n \int_{q_{i-1}}^{q_i} \varepsilon^2 \dx = (q_n - q_0)\varepsilon^2,
    \] 
    which implies that 
    \[
    \lim\limits_{n\to+\infty} \lVert H_n - F \rVert_{L^2(q_0,q_n)} = 0.
    \]
    Consequently, we obtain
    \[
        \lim\limits_{n,N\to\infty}  \lVert F - F_{N,n}\rVert_{L^2(q_0,q_n)} = 0,
    \]

    For the pointwise convergence property, we again invoke the auxiliary function $H_n$ and apply the triangle inequality: 
    \[
        |F_{N,n}(x) - F(x)| \leq |F_{N,n}(x) - H_n(x)| + |H_n(x) - F(x)|, \quad \forall x\in (q_0, q_n).
    \]
    Here we can also utilize the absolute continuity of $F$ to guarantee that
    \[
    \lim\limits_{n\to+\infty} |H_n(x) - F(x)| = 0, \quad \forall x\in (q_0, q_n).
    \]
    The pointwise convergence of $F_{N,n}$ to $H_n$ also follows from the fact that the discrepancy between $F_{N,n}(x)$ and $H_n(x)$ occurs only on the first and last subintervals, whose lengths vanish as $n\to +\infty$. Thus, for any fixed $x\in (q_0, q_n)$, we may choose $n$ sufficiently large so that $x$ lies in one of the interior subintervals.
    Putting these together, we obtain the pointwise convergence of $F_{N,n}$ to $F$ on $(q_0, q_n)$.
\end{proof}
Notably, most commonly used distributions fall within the class $\mathcal{X}$.\@ Typical examples include the Gaussian distribution, the exponential and gamma families, the lognormal distribution and the log-skew-normal distribution.

\subsection{Propagation capability of the model}

Our model is designed not only to approximate individual random variables, but also to preserve their essential characteristics under propagation through a DAG.\@ In particular, the model should remain stable under max and sum operations.

We begin by establishing results that justify the propagation capability in the $L^2$ sense. To avoid ambiguity, we denote by $F_k^{(l)}$ and $F_k^{(r)}$ the left-tail and right-tail CDFs of $X_k$ for $k=1,2$, respectively. Their expressions take the following form:

\begin{equation}\label{Eq:CDFtailform}
\left\{
\begin{aligned}
    F_k^{(l)}(x) &= \alpha_k^{(l)} \Phi\big(x; \mu_k^{(l)}, \sigma_k^{(l)}\big) + \sum\limits_{i=0}^{N_k - 1} \beta_{k,i}^{(l)} x^i \phi\big(x; \mu_k^{(l)}, \sigma_k^{(l)}\big), \\
    F_k^{(r)}(x) &= 1 - \alpha_k^{(r)} \left( 1 - \Phi\big(x; \mu_k^{(r)}, \sigma_k^{(r)}\big) \right) - \sum\limits_{i=0}^{N_k - 1} \beta_{k,i}^{(r)} x^i \phi\big(x; \mu_k^{(r)}, \sigma_k^{(r)}\big),
\end{aligned}
\right.    
\end{equation}
where $\alpha_k^{(l)}$, $\alpha_k^{(r)}$, $\beta_{k,i}^{(l)}$, and $\beta_{k,i}^{(r)}$ are real coefficients for $k = 1$, $2$ and $i = 0, 1, \dotsc, N_k - 1$. 

Let $F^{(l)}$ and $F^{(r)}$ denote the left-tail and right-tail CDFs of $X = X_1 \odot X_2$, respectively. For the sum case, $F^{(l)}$ and $F^{(r)}$ are both computed via convolution on the real line and thus share the same expression. Therefore, we omit the superscript $(l)$ and $(r)$ whenever no confusion arises.

Since the family $\{\e^{-x^2/2}x^n\}_{n\ge 0}$ forms a complete set in $L^2(\mathbb{R})$\cite[pp.~10, 108]{szeg1939orthogonal}, and we have established a constructive $L^2$-convergence result for the Hermite function expansion in Lem.~\ref{lem:HermtitepolyCompletenessL2param}, it suffices to establish the following lemma for $L^2$-justification of our tail model.

\begin{lemma}\label{lem:tailL2integrable}
    Suppose $X_1$, $X_2$ are two independent real-valued continuous random variables and have PDFs of the form~\eqref{eq:X12PDF}, and let $X = X_1 \odot X_2$ denote either their maximum or sum. Then for any $\delta > 0$, we have
    \begin{enumerate}[label=(\arabic*)]
        \item $F^{(l)}\in L^2(-\infty, q_0 + \delta)$, where $F^{(l)}$ denotes the left-tail CDF expression of $X$ that is naturally expanded from $(-\infty, q_0)$ onto $[q_0, q_0 + \delta)$.
        \item $1-F^{(r)}(x)\in L^2(q_n-\delta, +\infty)$, where $F^{(r)}$ denotes the right-tail CDF expression of $X$ that is naturally expanded from $(q_n, +\infty)$ onto $(q_n-\delta, q_n]$.
    \end{enumerate}
    In particular, $X\in \mathcal{X}$.
\end{lemma}

\begin{proof}
    \medskip
    \noindent(1) \textbf{Max case.} As $x \to -\infty$, $F^{(l)}(x) = F_1^{(l)}(x)F_2^{(l)}(x)$. In this case, we have 
    \begin{align*}  
        F_1^{(l)}(x)F_2^{(l)}(x)
        = \prod_{k=1}^2\left[\alpha_k^{(l)} \Phi\big(x;\mu_k^{(l)},\sigma_k^{(l)}\big) + \sum\limits_{i=0}^{N_k - 1}\beta_{k,i}^{(l)} x^i\phi\big(x;\mu_k^{(l)},\sigma_k^{(l)}\big)\right],
    \end{align*}
    where 
    \begin{equation}\label{eq:CDFTailAsymp}
        \alpha_k^{(l)} \Phi\big(x;\mu_k^{(l)},\sigma_k^{(l)}\big) + \sum\limits_{i=0}^{N_1 - 1}\beta_{k,i}^{(l)} x^i\phi\big(x;\mu_k^{(l)},\sigma_k^{(l)}\big) 
        \sim \left(-\frac{\alpha_k^{(l)}\sigma_k^{(l)}}{x-\mu_k^{(l)}} + \sum\limits_{i=0}^{N_1 - 1}\beta_{k,i}^{(l)} x^i\right)\phi\big(x;\mu_k^{(l)},\sigma_k^{(l)}\big) ,          
    \end{equation}
    as $x \to -\infty,\, k = 1, 2$. As the decay of $\e^{-x^2}$ suppresses all rational contributions, the decay behavior of $F_1(x)F_2(x)$ is dominated by the Gaussian factor.

    \textbf{Sum case.} We utilize the following union-bound-type inequalities:
    \[
    \begin{cases}
        \mathbb{P}(X_1 + X_2 < x) \leq \mathbb{P}\!\left(X_1 < \tfrac{x}{2}\right) + \mathbb{P}\!\left(X_2 < \tfrac{x}{2}\right), \\
        \mathbb{P}(X_1 + X_2 > x) \leq \mathbb{P}\!\left(X_1 > \tfrac{x}{2}\right) + \mathbb{P}\!\left(X_2 > \tfrac{x}{2}\right).
    \end{cases}
    \]
    Note that $\mathbb{P}(X_i < \frac{x}{2})$ also shows Gaussian decay as $x \to -\infty$, $i = 1$, $2$, it follows that $F_{X_1 + X_2}(x)$ is $L^2$-integrable on the left tail.

    \medskip
    \noindent(2) For the max case, as $x \to +\infty$, $F^{(r)}(x) = F_1^{(r)}(x)F_2^{(r)}(x)$. In this case, we have
    \begin{align*}  
        1-F_1^{(r)}(x)F_2^{(r)}(x)
        = 1-\prod_{k=1}^2\bigg[1-\alpha_k^{(r)} \left(1-\Phi\big(x;\mu_k^{(r)},\sigma_k^{(r)}\big)\right) - \sum\limits_{i=0}^{N_k - 1}\beta_{k,i}^{(r)} x^i\phi\big(x;\mu_k^{(r)},\sigma_k^{(r)}\big)\bigg].
    \end{align*}   
    Similar to the left-tail case, as $\alpha_k^{(r)} \left(1-\Phi\big(x,\mu_k^{(r)},\sigma_k^{(r)}\big)\right) + \sum\limits_{i=0}^{N_k - 1}\beta_{k,i}^{(r)} x^i\phi\big(x,\mu_k^{(r)},\sigma_k^{(r)}\big)$ shows Gaussian decay for $k = 1$, $2$ as $x \to +\infty$, their product also shows Gaussian-type decay behavior. Thus, $1-F(x)$ is $L^2$-integrable around $+\infty$.

    The proof for the sum case on the right tail follows the same reasoning as in the left-tail analysis.
\end{proof}
 
By combining the above lemma with Thm.~\ref{thm:modelExpressiveness}, we obtain the following result.
\begin{theorem}[Propagation capability of the model]\label{thm:modelPropagationCompleteness}
Suppose $X_1$, $X_2$ are independent real-valued continuous random variables and have PDFs of the form~\eqref{eq:X12PDF}, and let $X = X_1 \odot X_2$ denote either their maximum or sum, and write $F$ for the cumulative distribution function of $X$.
Then there exists $\{f_{N,n}\}$, each element of which follows a PDF of the form~\eqref{eq:modelPDF}, such that
    \begin{align*}
        &\lim_{n,N\to\infty}  \Big\Vert F(x) - \int_{-\infty}^{x}f_{N,n}(t)\dt\Big\Vert_{L^2(\mathbb{R})} = 0, \\
        &\lim_{n,N\to\infty}  \Bigl(F(x) - \int_{-\infty}^{x}f_{N,n}(t)\dt\Bigr)= 0, \qquad \forall x \in \mathbb{R}. 
    \end{align*}
\end{theorem}

\section{Algorithmic framework for delay graph propagation}\label{sec:imp}
In this section, we describe the algorithmic framework related to our delay model for propagating random delays on DAGs. As was discussed in Sec.~\ref{sec:model}, each random variable to be propagated, whose PDF follows the form~\eqref{eq:modelPDF} with a piecewise constant intermediate segment, can be represented by the following information: 
\begin{enumerate}[label=(\arabic*)]
    \item $\bigl\{c_j^{(l)}\bigr\}_{j=0}^N$, $\mu^{(l)}$, $\sigma^{(l)}$: the left-tail parameters,
    \item $\bigl\{\bigl(p_i, F(p_i)\bigr)\bigr\}_{i=0}^n$: the sample points and their CDF values in the intermediate segment,
    \item $\bigl\{c_j^{(r)}\bigr\}_{j=0}^N$, $\mu^{(r)}$, $\sigma^{(r)}$: the right-tail parameters.
\end{enumerate}

Given a DAG $G = (V, E)$, its topological order, and the incremental delay distribution at each node and along each edge, we aim to compute the arrival time distribution at each node. We assume that $N$ and $n$ are prescribed and remain fixed throughout the propagation. The incremental delays at each node and along each edge are also assumed to have pregiven analytic PDFs and CDFs that can be evaluated on each point on the real line.

Two main algorithmic components are involved in the propagation process:
\begin{enumerate}
    \item \textbf{CDF point sampling.} 
    We need to evaluate the CDF values at arbitrary points on the real line as well as compute the quantile values $q_0$ and $q_n$. This is not only required when we obtain the CDF sample points in the intermediate segment, but also when approximating the tail parameters.
    \item \textbf{Tail parameter fitting.}
    We employ the Levenberg--Marquardt (LM) algorithm~\cite{Marquardt1963LM} to fit the tail parameters based on the CDF values at several reference points. 
\end{enumerate}

In the delay propagation process, there are two scenarios where the two algorithms above are applied: 
\begin{enumerate}
    \item \textbf{Input projection.} 
    When a random variable $X$ is first introduced into the graph, it must be projected onto our model form. In this case, the analytic forms of its PDF and CDF are known a priori, allowing us to evaluate the CDF at any point on the real line directly.
    \item \textbf{Binary operation.}
    During propagation, max and sum operations are applied to pairs of random variables. This is carried out in a \textit{stage} manner. A \textit{stage} in the STA process typically refers to the timing arc from the input pin of a cell to the input pin of one of its fanout cells, which includes a cell delay arc and an interconnect delay arc. In this work, we adopt the simplified model in which all input-to-output paths of a cell share the same delay, i.e., a single node delay, although cases with distinct cell-delay arcs can be handled in the same way.
    As our model supports only binary operations, we here index the nodes and process in algebraic order by default. In practical scenarios, one may proceed in the order that the inputs arrive. 
\end{enumerate}

Notably, different initialization strategies are used for LM-based tail fitting in the two scenarios described above; these strategies are detailed in Subsec.~\ref{subsec:tailfitting}.

\subsection{CDF point sampling}
To determine the quantile values $q_0$ and $q_n$, we first apply the adaptive bracket expansion algorithm~\cite[Chapter~9]{press1989numerical}. Starting from an initial guess $x_\mathrm{init}$, we evaluate the CDF value $F(x_\mathrm{init})$ and determine whether the target CDF lies to its left or right. The interval is then expanded in the corresponding direction with an exponentially increasing step size until the function values at the two endpoints bracket the desired value. Once such a pair is identified, the standard bisection method is applied. After the interval $[l, r]$ becomes sufficiently small, linear interpolation is used to obtain the final solution. The full procedure is summarized in Alg.~\ref{alg:bracket-bisect}. 

\begin{algorithm}[ht]
\caption{Adaptive bracket expansion and bisection for computing quantile values}
\label{alg:bracket-bisect}
\begin{algorithmic}[1]
\Require CDF $F(\cdot)$, target CDF value $y\in(0,1)$, initial guess $x_0$, initial step size $\Delta x>0$, expansion factor $\eta$, tolerance $\varepsilon>0$
\Ensure Approximate solution $x$ satisfying $\lvert F(x)-y\rvert \le \varepsilon$

\State $F_0 \gets F(x_0)$
\If{$F_0 = y$}
    \State \Return $x_0$
\EndIf
\State $s \gets \mathrm{sign}(y-F_0)$
\Repeat
    \State $x_1 \gets x_0 + s\,\Delta x$, \ $F_1 \gets F(x_1)$
    \If{$(F_1-y)(F_0-y) > 0$}
        \State $x_0 \gets x_1$, \ $F_0 \gets F_1$, \ $\Delta x \gets \eta\,\Delta x$
    \EndIf
\Until{$(F_1-y)(F_0-y) \le 0$}

\State $l \gets \min(x_0,x_1)$, \ $r \gets \max(x_0,x_1)$
\While{$r-l>\varepsilon$}
    \State $m \gets (l+r)/2$
    \If{$F(m)\le y$}
        \State $l \gets m$
    \Else
        \State $r \gets m$
    \EndIf
\EndWhile

\State \Return $\dfrac{(F(r)-y)\,l + (y-F(l))\,r}{F(r)-F(l)}$
\end{algorithmic}
\end{algorithm}

After obtaining $q_0$ and $q_n$, we uniformly sample $n-1$ more points in $[q_0, q_n]$ and evaluate the CDF values at these points. The full procedure is summarized in Alg.~\ref{alg:CDFpoints}.

If a closed-form expression for the input distribution is available and enables direct evaluation of quantiles, this sampling step can be omitted. For example, in many practical settings, delays are modeled as Gaussian or lognormal random variables, for which quantiles can be computed analytically. Moreover, once a random variable has been projected onto our model, one may also recover a parametric distribution using moment matching based on numerically estimated moments.

When evaluating the CDF at the desired points, two different scenarios arise. If the distribution admits a known analytic expression, the values can be computed directly. Otherwise, numerical evaluation is required. In our framework, the max operation falls into the former category, whereas the sum operation belongs to the latter. For the numerical evaluation of the convolution integral in the sum case, the integral is decomposed into three regions: the left tail, the intermediate segment, and the right tail.
In each region, rather than relying entirely on general-purpose numerical integrators such as GSL's \texttt{qagi}~\cite{gough2009gnu} or Boost's \texttt{sinh\_sinh}, we exploit the polynomial-Gaussian tail structure and the piecewise linear approximation of the CDF in the intermediate segment to obtain closed-form expressions whenever possible, and resort to numerical quadrature methods only for the remaining parts.

\begin{algorithm}[ht]
\caption{CDF point sampling}
\label{alg:CDFpoints}
\begin{algorithmic}[1]
\Require Target CDF values $y_0$, $y_n$, number of CDF points $n$, the CDF mapping $F$ 
\Ensure $\{\bigl(q_i, F(q_i)\bigr)\}_{i=0}^n$

\State Apply Alg.~\ref{alg:bracket-bisect} to find $q_0$ and $q_n$ such that $F(q_0) = y_0$, $F(q_n) = y_n$

\For{$i = 1$ to $n-1$}
    \State $q_i \gets q_0 + \frac{i}{n} (q_n-q_0)$
    \State Compute $F(q_i)$ according to the analytic or numerical method
\EndFor

\end{algorithmic}
\end{algorithm}
\subsection{Tail parameter fitting}\label{subsec:tailfitting}
After obtaining the CDF values at the reference points, we proceed to fit the tail parameters. If the input distribution is Gaussian, we directly set the tail parameters accordingly and skip the fitting step.

Although we have already illustrated the convergence properties of the Hermite function expansion for our model in Sec.~\ref{sec:complete}, we do not adopt this expansion in practice. The reasons are twofold. First, computing Hermite coefficients is numerically expensive, especially for the sum operation, which requires evaluating integrals over the entire real line. Second, the intermediate segment lacks sufficient smoothness to guarantee any convergence rate estimate, potentially leading to slow convergence in practice.

For a general tail-fitting problem, the determination of the tail parameters is formulated as a nonlinear least-squares (NLS) problem and we solve it by the Levenberg--Marquardt (LM) algorithm.
In the following discussion, we focus on the left-tail case; the corresponding right-tail procedure is presented in Appendix~\ref{AppendixA}. Let $\{x_i, y_i\}_{i=1}^k$ denote the reference CDF points, where $y_i = F(x_i)$. These points do not necessarily satisfy $x_i < q_0$ for all $i=1, \dotsc, k$. The CDF values $y_i$ at these points can be numerically computed. The selection of reference points is flexible in practice. In our implementation, we choose $k$ equally spaced points in the interval $[F^{-1}(q_0/4), F^{-1}(2q_0)]$ for the left tail, where $F^{-1}$ denotes the quantile function of the target distribution. The right tail reference points are chosen in a similar manner. The target function in the LM algorithm is set as 
\begin{equation}\label{eq:LMTargetFunc}
    \min \sum_{i=1}^k \biggl[ \Bigl( \widetilde{F}\bigl(x_i; \{c_j^{(l)}\}_{j=0}^N, \mu^{(l)}, \sigma^{(l)}\bigr) - y_i \Bigr) \Big/ y_i \biggr]^2 + \lambda \bigl(c_0^{(l)}-1\bigr)^2 + \lambda\biggl( \sum\limits_{j = 1}^N \bigl(c_j^{(l)}\bigr)^2\biggr),
\end{equation}
where the function $\widetilde{F}$ is the approximate left-tail CDF  that can be directly computed from the parameters $\{c_j^{(l)}\}_{j=0}^N$, $\mu^{(l)}$, and $\sigma^{(l)}$.

To balance the contributions of reference points across the tail region, we adopt the relative error formulation in Eq.~\eqref{eq:LMTargetFunc}, which assigns comparable importance to points near the tail and those moderately farther away. A regularization term is introduced to prevent overfitting. Without it, the optimization may produce excessively large polynomial coefficients, which in turn can make the propagation process unstable.

Since the Jacobian matrix $J$ can be directly derived and evaluated in our model, we can utilize the LM algorithm, which solves 
\[(J^T J + \lambda_{\mathrm{LM}} I) \delta x = -J^T r,\]
where $\lambda_{\mathrm{LM}}$ is a damping factor, $I$ is the identity matrix and $r$ is the residual vector defined by 
\[
r_i = 
\begin{cases}
\displaystyle \frac{y_i - \widetilde{F}\bigl(x_i; \{c_j^{(l)}\}_{j=0}^N, \mu^{(l)}, \sigma^{(l)}\bigr)}{y_i}, & i = 1, \dotsc, k, \\
\displaystyle \sqrt{\lambda} \, c_{i - k - 1}^{(l)}, & i = k+1, \dotsc, k + N + 1.
\end{cases}
\]

To impose the constraint that $\sigma^{(l)} > 0$, we reparameterize the scale as $p_\sigma = \log(\sigma^{(l)})$. This transformation keeps the objective and residual structure intact but alters the Jacobian accordingly.

As an illustrative example, consider the case $N = 2$. For simplicity, superscripts are omitted and the left-tail PDF is written as 
\begin{equation}\label{eq:orderTwoEg}
\tilde{f}(x) = (\alpha + \beta x + \gamma x^2)\phi(x;\mu,\sigma).
\end{equation} 
The corresponding left-tail CDF is given by
\[
\widetilde{F}(x) = \bigl(\beta \mu + \alpha + \gamma (\mu^2 + \sigma^2)\bigr) \, \Phi\left( x;\mu,\sigma \right) - (\beta + \gamma x + \gamma \mu) \sigma^2 \, \phi\left( x;\mu,\sigma \right).
\]
The scale parameter is reparameterized as $\sigma = \exp(p_\sigma)$ to ensure positivity during optimization. Correspondingly, the optimization variable is $p_\sigma = \log(\sigma)$.

Let $\Phi_i = \Phi(x_i; \mu, \sigma)$, $\phi_i = \phi(x_i; \mu, \sigma)$ and $z_i = (x_i-\mu)/\sigma$, then the $i$-th row of the Jacobian matrix $J$ is given by
\[
J(i,:) = -\frac{1}{y_i}
\begin{bmatrix}
\mu \, \Phi_i - \sigma^2 \, \phi_i \\
\Phi_i \\
(\mu^2 + \sigma^2) \Phi_i - (x_i + \mu)\sigma \phi_i \\
(\beta + 2\gamma \mu) \Phi_i
- \left(
\alpha + \beta \mu + \gamma(\mu^2 + \sigma^2)
+ \gamma \sigma^2
+ \frac{(x_i - \mu)(\beta + \gamma(x_i + \mu))}{\sigma}
\right) \sigma\phi_i \\
\sigma \left[
2\gamma \Phi_i +
\left(
-(\beta + \gamma(x_i + \mu))(z_i^2 + 1)
- \frac{z_i}{\sigma}(\beta \mu + \alpha + \gamma(\mu^2 + \sigma^2))
\right) \sigma\phi_i
\right]
\end{bmatrix}^\top, 
\]
for $i = 1, \dotsc, k$. 

The regularization term contributes $N+1$ additional rows to the Jacobian matrix. For $N = 2$, the last three rows are
\begin{align*}
J(k+1,:) &= \begin{bmatrix} \sqrt{\lambda}, & 0, & 0, & 0, & 0 \end{bmatrix}, \quad\\
J(k+2,:) &= \begin{bmatrix} 0, & \sqrt{\lambda}, & 0, & 0, & 0 \end{bmatrix}, \quad\\
J(k+3,:) &= \begin{bmatrix} 0, & 0, & \sqrt{\lambda}, & 0, & 0 \end{bmatrix}.
\end{align*}
This ensures that minimizing the total squared residual norm corresponds exactly to the regularized objective function in Eq.~\eqref{eq:LMTargetFunc}.

Putting everything together, we obtain Alg.~\ref{alg:tail-fitting} for fitting the tail parameters.
\begin{algorithm}[htbp]
\caption{Tail parameter fitting via the Levenberg--Marquardt algorithm}
\label{alg:tail-fitting}
\begin{algorithmic}[1]
\Require Reference points $\{(x_i, y_i)\}_{i=1}^k$, polynomial degree $N$, regularization weight $\lambda$
\Ensure Fitted parameters $x = \left( c_0,\, c_1,\, \dotsc,\, c_N,\, \mu,\, \sigma \right)^\top$

\State Initialize $\mu$, $\sigma$ depending on the operation type and tail side
\State Reparameterize $\sigma$ as $p_\sigma = \log(\sigma)$
\State Initialize coefficients $c_0, c_1, \dots, c_N$.

\While{not converged}
    \For{$i = 1$ to $k$}
        \State Compute residual $r_i = \dfrac{\tilde{F}(x_i) - y_i}{y_i}$
    \EndFor

    \For{$j = 0$ to $N$}
        \State Add regularization residual $r_{k + j + 1} = \sqrt{\lambda} \cdot c_j$
    \EndFor
    \State Compute the Jacobian matrix $J$
    \State Solve LM step: $(J^\top J + \lambda_{\mathrm{LM}} I)\, \delta x = -J^\top r$
    \State Update $x \gets x + \delta x$, with $\sigma = \exp(p_\sigma)$ and $\lambda_{\mathrm{LM}}$
\EndWhile
\end{algorithmic}
\end{algorithm}

It is noteworthy that we do not impose the constraint that the CDF values at the quantile points $q_0$ and $q_n$ must remain fixed when formulating the NLS problem. Since our primary objective is to achieve an accurate overall fit across the tail region, allowing small deviations at $q_0$ and $q_n$ has negligible impact on the resulting approximation.

As the LM algorithm only ensures local convergence, we aim to provide a good initialization to achieve reliable and stable optimization performance. 

During input projection, the tail parameters are initialized by fitting either the left or right tail to a Gaussian via least squares. Given sample points $\{(x_i, y_i)\}_{i= 1}^{k}$, we compute $z_i = \Phi^{-1}(y_i)$ and then fit $x_i \approx \mu + \sigma z_i$ by solving 
\[
\min_{\mu, \sigma} \sum\limits_{i = 1}^k (x_i - \mu - \sigma z_i)^2,
\]
which admits the following closed-form solution:
\[
\left\{
\begin{aligned}
\sigma &= \frac{\sum_{i=1}^k (z_i-\bar z)(x_i-\bar x)}{\sum_{i=1}^k (z_i-\bar z)^2},\\
\mu &= \bar x - \sigma\,\bar z,
\end{aligned}
\right.
\qquad
\text{where}\quad
\bar x=\frac1k\sum_{i=1}^k x_i,\;
\bar z=\frac1k\sum_{i=1}^k z_i.
\]

The initialization of the tail parameters in binary operations differs from that used in input projection. Since the target distribution is now obtained through either a max or a sum operation applied to random variables of a prescribed form, additional information about the asymptotic behavior of the tails becomes available. This allows us to devise more informed and effective initialization strategies.

When initializing for the sum operation, we utilize the closure of Gaussian random variables under sum operations. For the left tail, we initialize $\mu$ and $\sigma$ in Eq.~\eqref{eq:orderTwoEg} as
\[
\mu = \mu_1^{(l)} + \mu_2^{(l)}, \quad \sigma^2 = \bigl(\sigma_1^{(l)}\bigr)^2 + \bigl(\sigma_2^{(l)}\bigr)^2.
\]
The same initialization strategy is used for the right tail.

When initializing for the max operation, it is important to observe that the condition $F_1F_2\to 0$ does not necessarily imply that both $F_1$ and $F_2$ are simultaneously small.
Hence, the initialization of the left tail should distinguish between the following two cases:
\begin{enumerate}
\item If the values of $F_1$ and $F_2$ differ significantly at the chosen reference points, one of them may remain of non-negligible magnitude even when the other approaches zero. In this case, the left-tail initialization is taken from the random variable whose CDF values are smaller at those points.

\item If the values of $F_1$ and $F_2$ are comparable at the reference points, we assume that both $F_1$ and $F_2$ show asymptotic behaviors.
The left-tail exact PDF of the sum should satisfy: 
\[
    f^{(l)}(x) = F_1^{(l)}(x)f_2^{(l)}(x) + f_1^{(l)}(x)F_2^{(l)}(x). 
\]
Suppose that the input PDFs are of the form 
$
f_i^{(l)}(x) = p_i(x)\phi(x;\mu_{i}^{(l)}, \sigma_{i}^{(l)})$, where $ i = 1, 2,$ and $p_i(x)$ are polynomials. As $x\to -\infty$, we can see from Eq.~\eqref{eq:CDFTailAsymp} that $F_1^{(l)}$ and $F_2^{(l)}$ can both be seen as the similar form, since the fraction term in the bracket shows a decay behavior as $x\to -\infty$ and can be omitted. Thus we write 
\[
F_i^{(l)}(x) \sim q_i(x)\phi(x;\mu_{i}^{(l)}, \sigma_{i}^{(l)}), \quad i = 1, 2,\quad \mathrm{as}\; x\to -\infty,
\]
for some polynomials $q_i(x)$. Completing the square in the quadratic exponent of the product $\phi(x;\mu_{1}^{(l)}, \sigma_{1}^{(l)})\phi(x;\mu_{2}^{(l)}, \sigma_{2}^{(l)})$ yields the initialization
\[
\left\{
\begin{aligned}
\mu^{(l)} &= \Bigl( \tfrac{\mu_1^{(l)}}{(\sigma_1^{(l)})^2} + \tfrac{\mu_2^{(l)}}{(\sigma_2^{(l)})^2} \Bigr)(\sigma^{(l)})^2, \\
\sigma^{(l)} &= \Bigl( (\sigma_1^{(l)})^{-2} + (\sigma_2^{(l)})^{-2} \Bigr)^{-1/2}.
\end{aligned}
\right.
\]
\end{enumerate}
We adopt the following heuristic criterion for determining when the values of $F_1$ and $F_2$ ``differ significantly'' as: 
\[
F_1(q_{1,0}) > F_2(q_{2,n}) \quad\text{or}\quad
F_2(q_{2,0}) > F_1(q_{1,n}).
\]

Right-tail initialization is comparatively simpler, because both $F_1$ and $F_2$ are close to $1$ when $F_1\cdot F_2\to 1$. The right-tail PDF behavior of $\max\{X_1, X_2\}$ is dominated by the component of the larger $\sigma^{(r)}$ value, whose decay rate is slower. 
In practice, however, the reference points used for tail fitting do not lie sufficiently far into the extreme tail. Consequently, the decaying behavior among these points may not fully reflect the asymptotic characteristics.

Therefore, rather than relying completely on the asymptotic behavior, we adopt a heuristic Gaussian-based initialization strategy as follows: 
\[
\left\{
\begin{aligned}
\mu^{(r)} &= \max\bigl(\mu^{(r)}_1,\ \mu^{(r)}_2\bigr), \\
\sigma^{(r)} &= \frac{1}{3}\left(\max\bigl\{\mu^{(r)}_1 + 3\sigma^{(r)}_1,\ \mu^{(r)}_2 + 3\sigma^{(r)}_2\bigr\} - \mu^{(r)}\right).
\end{aligned}
\right.
\]

After determining the initial values of $\mu$ and $\sigma$, the polynomial coefficients are initialized via a least-squares fitting method. Evaluating the PDF at the reference points gives the values $\{p(x_i)\}_{i=0}^k$, and the coefficients are obtained by applying \texttt{polyfit} to solve
\begin{equation}\label{eq:polyfitTargetFunc}
\min_{p\in P_N} \sum\limits_{i = 1}^k \left[p(x_i)-\dfrac{\tilde{f}(x_i)}{\phi(x_i;\mu,\sigma)}\right]^2.
\end{equation}

After each binary operation, the values of $\widetilde{F}$ at $q_0$ and $q_n$ may deviate from the target values $F(q_0)$ and $F(q_n)$. But we keep these values unchanged, as such deviations has limited impact in the propagation process and are consistent with our theoretical results.

\subsection{Overall workflow}\label{subsec:workflow}
We summarize below the overall workflow that integrates all components discussed in this section into a complete delay propagation procedure on a DAG.\@

Given a DAG and all the random variables representing node delays and edge delays, we first perform \emph{input projection} for each input random variable, converting its distribution into left-tail parameters, right-tail parameters, and CDF sample points in the intermediate segment. We then topologically sort the graph to determine the propagation order. Starting from the source node(s), propagation proceeds stage by stage along this order, during which the required binary operations are carried out. 

At each node, we first obtain its arrival time by combining the outputs of all its predecessor nodes through successive max operations. The node delay is then added via a single sum operation to produce the node's output delay. Finally, the corresponding edge delays are added when passing this output delay for further propagation. 

To complement this workflow, Fig.~\ref{fig:delay_stage} illustrates a complete propagation stage on a simple DAG consisting of two source nodes and one sink node. For each incoming edge $(i,3)$, the edge delay $d_{i,3}$ is first projected onto our model space, yielding $\theta_{i,3}$. 
The predecessor node delays $d_1$ and $d_2$ are also projected to $\theta_1$ and $\theta_2$. Each pair $(\theta_i, \theta_{i,3})$ is combined through a \texttt{sum} operation to produce the arrival distributions $\theta_{3,\mathrm{in}1}$ and $\theta_{3,\mathrm{in}2}$. 
A \texttt{max} operation merges these two candidates into $\theta_{3,\mathrm{in}}$, after which the node delay $d_3$ is projected to $\theta_3$ and added, giving the output distribution $\theta_{3,\mathrm{out}}$.

\begin{figure}[ht]
    \centering
    \begin{subfigure}[t]{0.4\textwidth}
        \centering
        \includegraphics[width=0.7\linewidth,height=4cm]{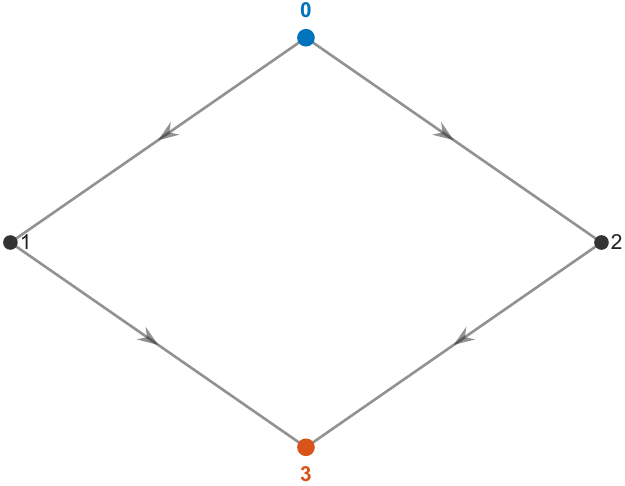}
        \caption{DAG topology}
        \label{fig:delay_stage_process}
    \end{subfigure}
    \hfill
    \begin{subfigure}[t]{0.58\textwidth}
        \centering
        \includegraphics[width=\linewidth,height=6cm]{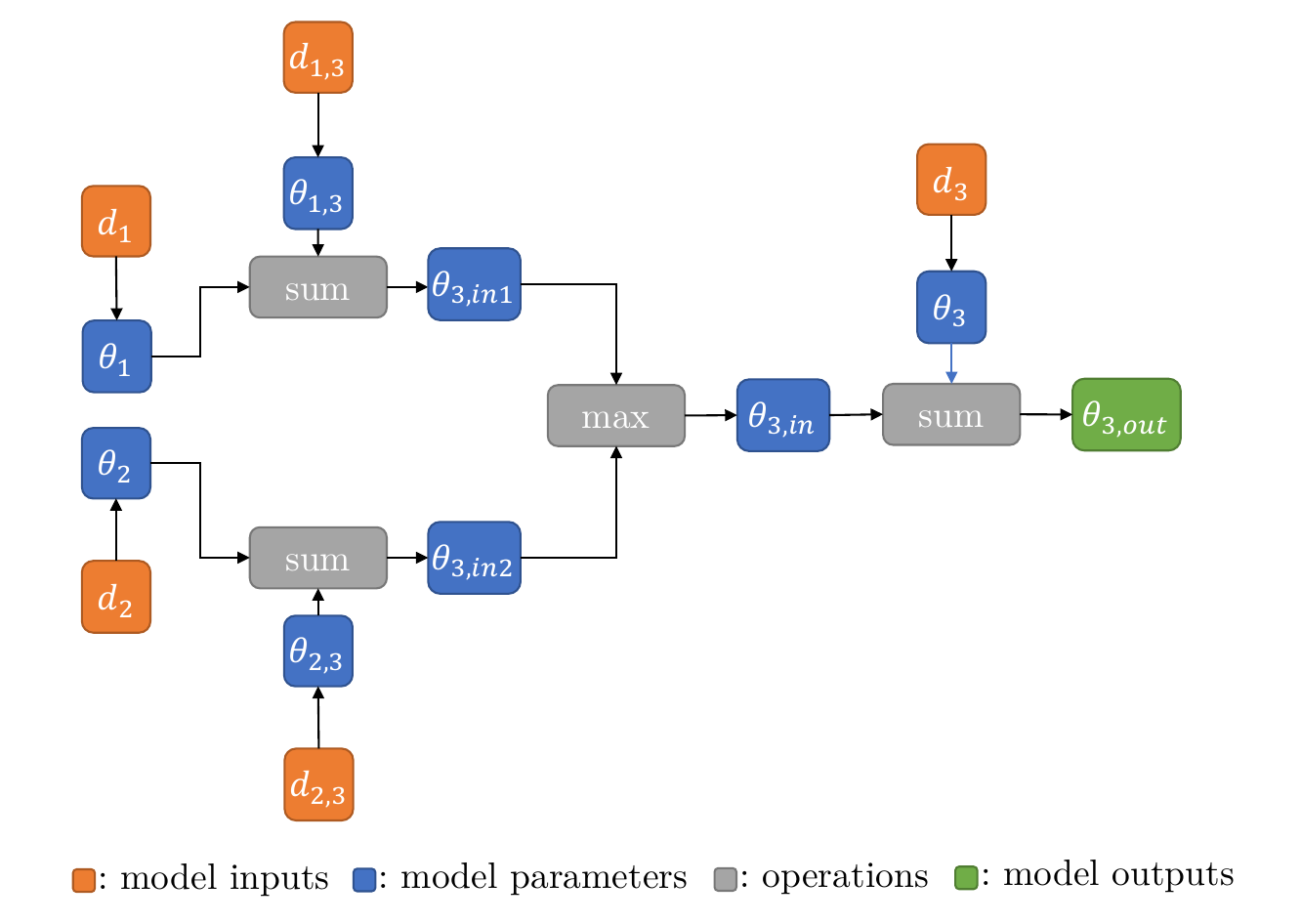}
        \caption{Delay propagation workflow}
        \label{fig:delay_stage_DAG}
    \end{subfigure}

    \caption{Illustration of delay propagation with a three-node DAG}
    \label{fig:delay_stage}
\end{figure}

\section{Numerical results}\label{sec:numeric}
In this section, we present numerical experiments to validate our proposed model and algorithms. The experiments are designed to assess both the approximation accuracy of the three-segment representation and its stability under repeated sum and max operations on DAGs. In particular, we focus on the accuracy of tail quantiles, which are of primary interest in STA.\@ 

\subsection{Experiment setup}
To evaluate the performance of our proposed algorithm in Sec.~\ref{sec:imp} and to assess the feasibility of our proposed model in Sec.~\ref{sec:model}, we conducted numerical experiments under different scenarios. Throughout all experiments, we set the number of CDF sample points $n = 100$, the degree of the tail polynomials $N = 2$. On each tail, $k = 21$ points are used for tail fitting, and the left and right quantiles are respectively set to $0.135\%$ and $99.865\%$, corresponding to the $\pm 3\sigma$ points of a standard Gaussian distribution. To focus on the behavior of the propagation algorithm itself, we adopt a simplified setting. Instead of letting delays at later stages depend on those of preceding stages, we assume that all node and edge delays in the graph are predetermined random variables. In other words, all delay information required for propagation is assumed to be available from the beginning.

We consider three experimental scenarios, each associated with a target distribution used as the primary reference for comparison. In graph propagation experiments, the distribution of the sink node is taken as the distribution of interest.
To quantitatively assess the accuracy of different methods, we compare the key quantiles of our desired distributions, including the $0.135\%$, $1\%$, $99\%$ and $99.865\%$ quantiles, among which the $0.135\%$ and $99.865\%$ quantiles are of particular importance. The quantile errors are defined as their relative error compared to the ground truth, i.e.,
\begin{equation}\label{eq:quantileError}
 \mathrm{Err}(q) 
    = \frac{\lvert q - q_{\mathrm{true}} \rvert}{q_{\mathrm{true}}},    
\end{equation}
where $q_{\mathrm{true}}$ denotes the ground-truth value.

For each scenario, we compare some of, if not all, the ground truth, the output of our propagation model, the Gaussian moment-matching results whose moments are derived from the ground truth, and an additional Gaussian-based approximation described below.

The additional baseline is a simple quantile-preserving method that provides a fast but crude delay-propagation approximation. For the sum operation, it exploits the closure of independent Gaussian random variables:
\[
\begin{cases}
\mu_{\mathrm{sum}} &= \mu_1 + \mu_2, \\
\sigma_{\mathrm{sum}} &= \sqrt{\sigma_1^2 + \sigma_2^2}.
\end{cases}
\]
For the max operation, this method computes
\[
\begin{cases}
    \mu_{\max} &= \max\{\mu_1, \mu_2\}, \\
    \sigma_{\max} &= \tfrac{1}{3}\bigl(\max\{\mu_1 + 3\sigma_1, \mu_2 + 3\sigma_2\} - \mu_{\max}\bigr).
\end{cases}
\]
We refer to this method as the \emph{Gaussian quantile-preserving method}.

For each case, the ground truth is obtained either by Monte Carlo simulation or directly from the analytical form of the underlying distribution class. Although our model is inherently CDF-oriented, we compare both the PDF and the CDF of the methods to provide a more comprehensive evaluation.

To ensure the effectiveness of our proposed model, we compare it against two alternative tail construction strategies, while retaining the same representation in the intermediate segment. 
The first strategy defines the left-tail PDF as
\[
f(x) = 
\begin{cases}
    0, &\quad x < \dfrac{F(q_1)q_0 - F(q_0)q_1}{F(q_1) - F(q_0)},\\[1.2ex]
    \dfrac{F(q_1) - F(q_0)}{q_1 - q_0}, &\quad \dfrac{F(q_1)q_0 - F(q_0)q_1}{F(q_1) - F(q_0)} \leq x < q_0.
\end{cases}
\]
This construction corresponds to linearly extrapolating the CDF on $[q_0, q_1]$ towards $-\infty$ until the extrapolated line reaches $0$. Similarly, the right tail is constructed by linearly extending the CDF on $[q_{n-1}, q_n]$ towards $+\infty$ until it reaches $1$.
We refer to this approach as the \emph{extrapolation method}. 

The second strategy adopts a different idea by introducing two fixed sentinel points $q_{-1}$ and $q_{n+1}$, where $F(q_{-1}) = 0$ and $F(q_{n+1}) = 1$. These points serve as global boundaries for all distributions in the graph.
The left-tail CDF is obtained by linearly interpolating between $(q_{-1}, 0)$ and $(q_0, F(q_0))$, while the right-tail CDF is obtained by interpolating between $(q_n, F(q_n))$ and $(q_{n+1}, 1)$. We refer to this strategy as the \emph{sentinel method}.

\subsection{Lognormal projection}
In this scenario, we evaluate the performance of our projection algorithm on a lognormal random variable $X\sim \mathrm{LN}(0,1)$.
The PDF of $\mathrm{LN}(0,1)$ is given by 
\[
f_{\mathrm{LN}}(x) = \frac{1}{x\sqrt{2\pi}} \exp\left(-\frac{(\ln x)^2}{2}\right), \quad x > 0.
\]
The ground truth is set as the exact PDF and CDF of $\mathrm{LN}(0,1)$, and our model is used to fit the corresponding tail parameters. The Gaussian moment-matching approximation $\mathcal{N}(\mu_{\mathrm{LN}}, \sigma_{\mathrm{LN}}^2)$ is given by
\[
\mu_{\mathrm{LN}} = \exp\Bigl(\mu + \frac{\sigma^2}{2}\Bigr), \quad \sigma_{\mathrm{LN}}^2 = \bigl[\exp(\sigma^2) - 1\bigr]\exp(2\mu + \sigma^2),
\] 

The results are displayed in Fig.~\ref{fig:ln_cdf_pdf_compare}. 
\begin{figure}[ht]
    \centering
    \begin{subfigure}{0.48\textwidth}
        \centering
        \includegraphics[width=\linewidth]{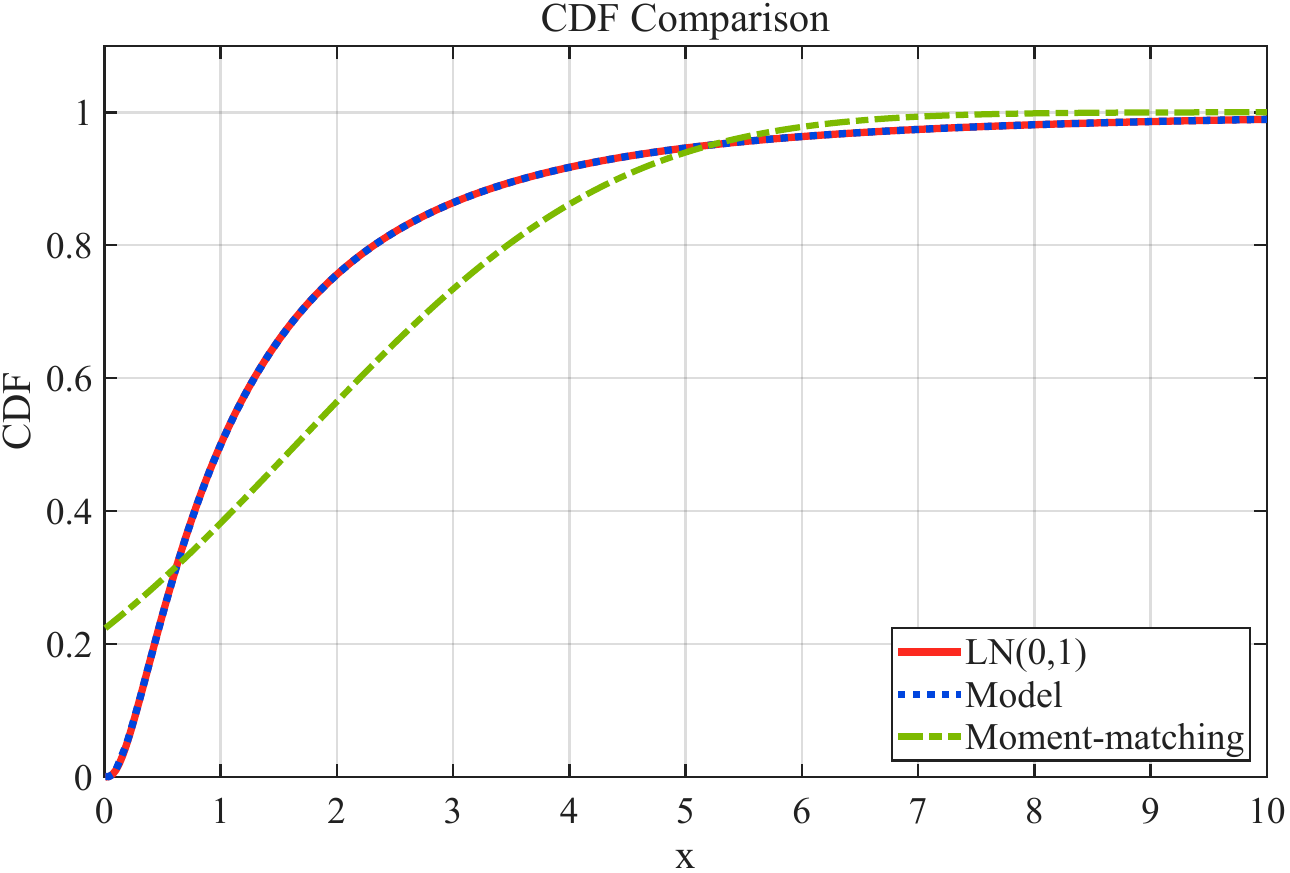}
        \caption{CDF comparison}
        \label{fig:ln_cdf}
    \end{subfigure}
    \hfill
    \begin{subfigure}{0.48\textwidth}
        \centering
        \includegraphics[width=\linewidth]{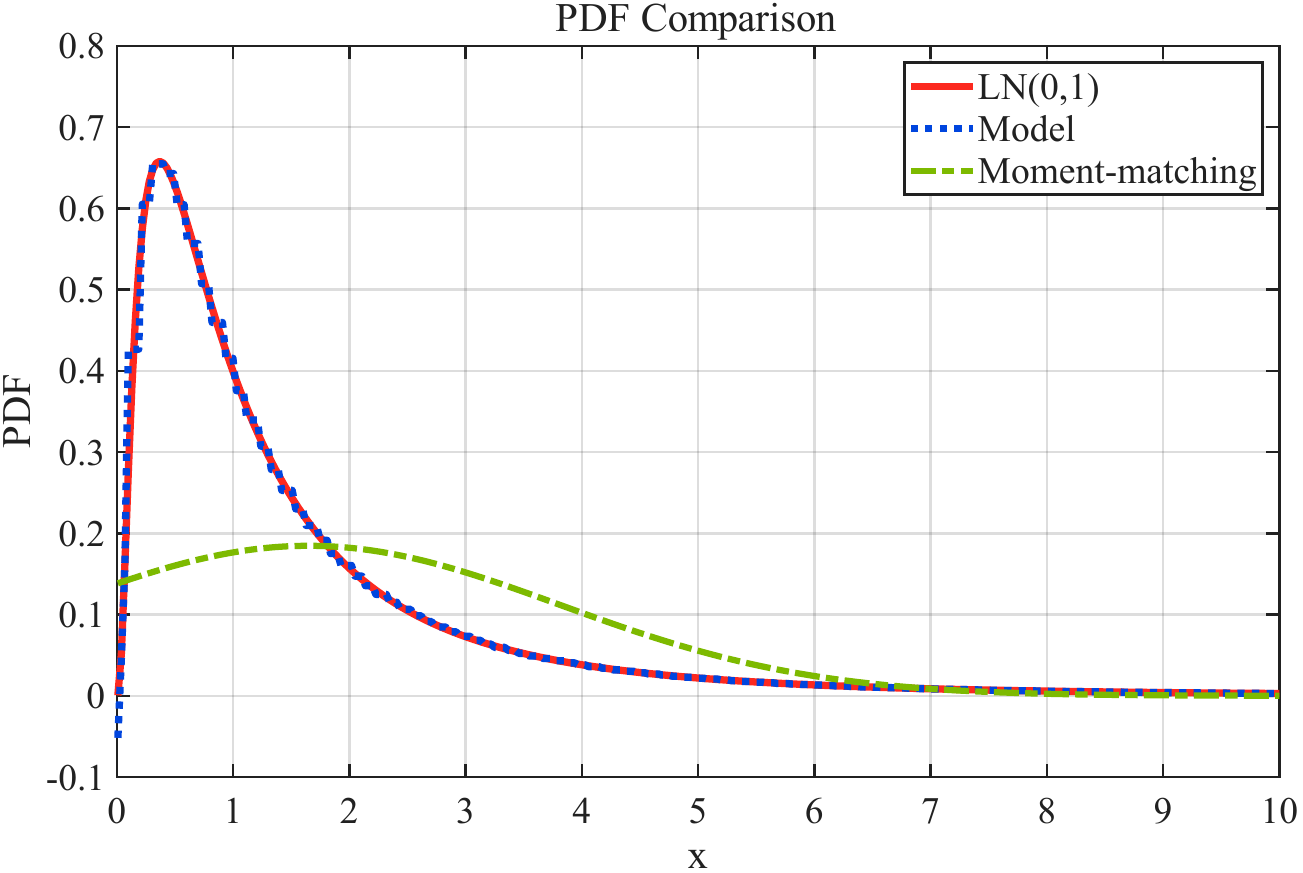}
        \caption{PDF comparison}
        \label{fig:ln_pdf}
    \end{subfigure}
    \caption{   Comparison of \subref{fig:ln_cdf} CDF and \subref{fig:ln_pdf} PDF between $\mathrm{LN}(0,1)$, our model results, and Gaussian results using the moment-matching method.  }
    \label{fig:ln_cdf_pdf_compare}
\end{figure}
The results demonstrate that our model provides an excellent approximation to the ground truth across the entire distribution. Although some deviations appear in the PDF comparison, these differences have minimal impact on quantile propagation and are therefore acceptable. 
In contrast, the Gaussian moment-matching approximation exhibits substantial errors, as expected from the strong skewness of the lognormal distribution. This demonstrates that our model effectively captures higher-order moment information and accurately projects lognormal inputs into our model space.
\subsection{Inverted binary tree propagation}
To evaluate the performance of our algorithm in delay propagation, we construct an inverted binary tree of $5$ levels. The node $0$ at level $0$ is the source of the graph, where no input delay is imposed. It distributes delays to the nodes at level $1$. At each node starting from level $2$, the output distribution is obtained by taking the maximum of the two incoming distributions and then adding its own incremental delay. This construction guarantees that, at every merge point in the graph, the incoming delays are independent. The structure of the graph is shown in Fig.~\ref{fig:inverted_binary_tree}.

\begin{figure}[ht]
    \centering
    \includegraphics[width=0.5\linewidth]{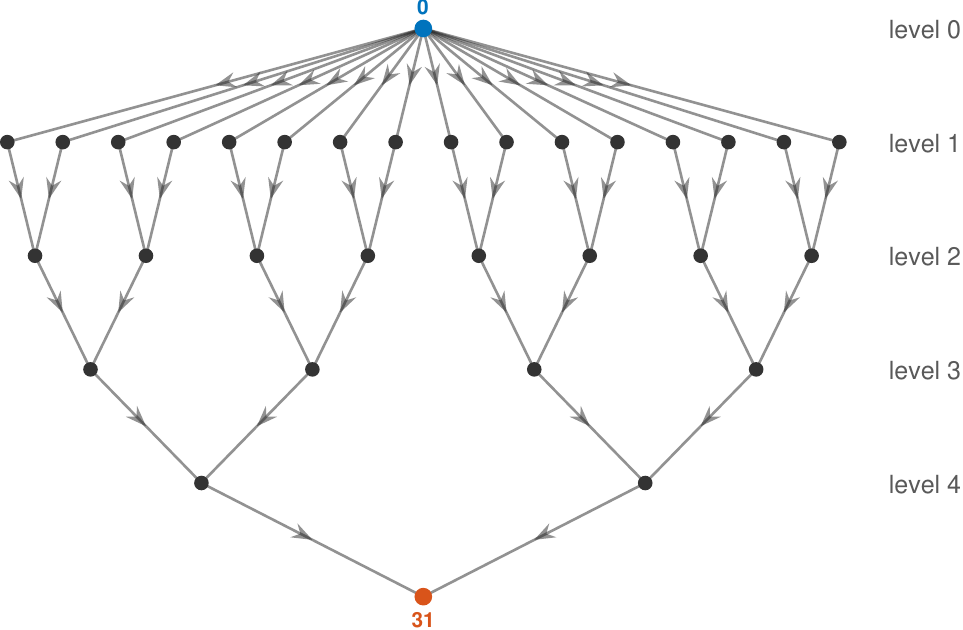}
    \caption{Inverted binary tree structure}
    \label{fig:inverted_binary_tree}
\end{figure}

In this experiment, each edge delay is modeled as a Gaussian random variable, whose mean $\mu$ is uniformly sampled from $[3.0,5.0]$ and standard deviation $\sigma$ is uniformly sampled from $[1.5, 2.0]$. 

Our primary interest is the distribution at the sink node, i.e., node $31$. The ground truth is obtained by a Monte Carlo simulation with $10^7$ samples. 

The resulting CDF and PDF comparisons are presented in Fig.~\ref{fig:bt_cdf_pdf_compare}, and the corresponding quantile errors are listed in Table~\ref{tab:quantile-bt}.

\begin{figure}[t]
  \centering

  \begin{subfigure}[t]{0.48\textwidth}
    \centering
    \includegraphics[width=\linewidth]{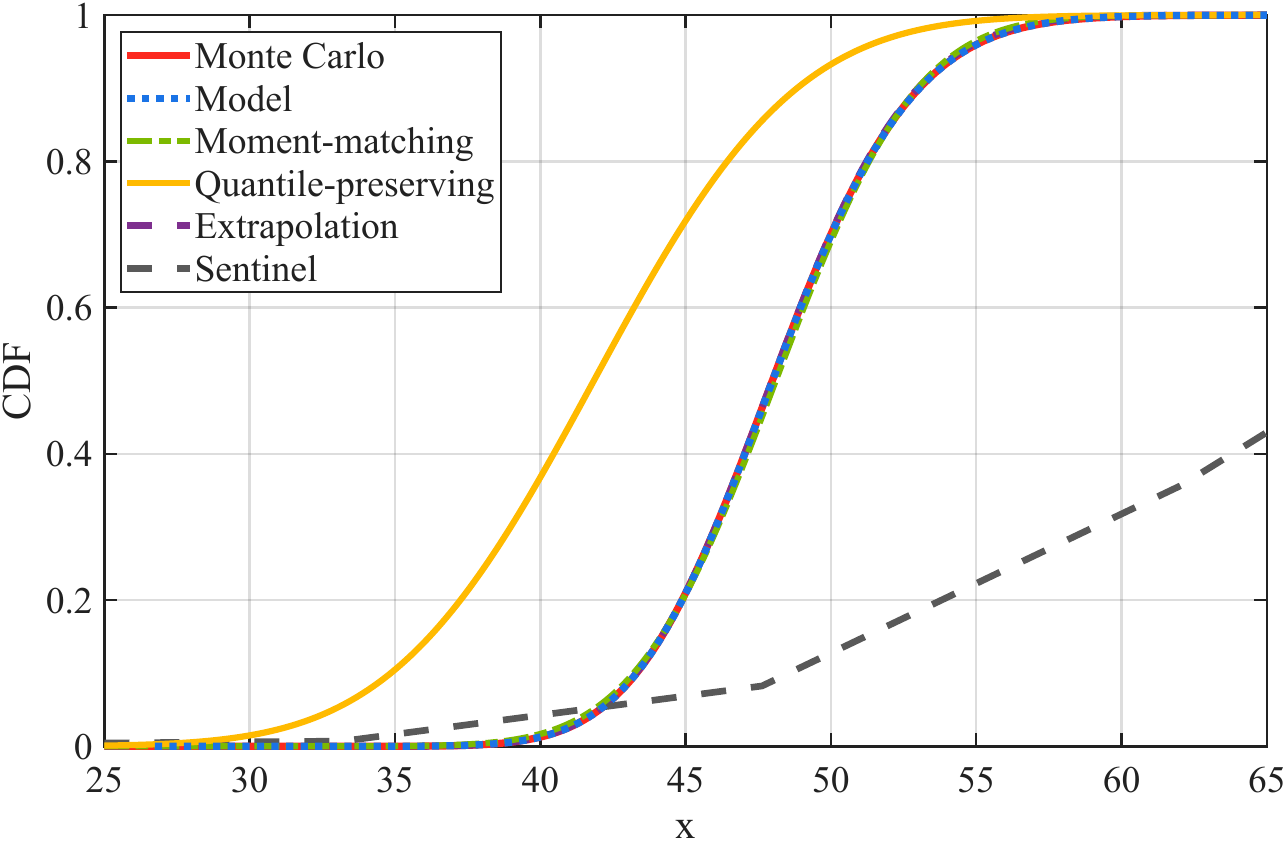}
    \caption{raw CDF}
    \label{fig:bt_cdf_raw}
  \end{subfigure}
  \hfill
  \begin{subfigure}[t]{0.48\textwidth}
    \centering
    \includegraphics[width=\linewidth]{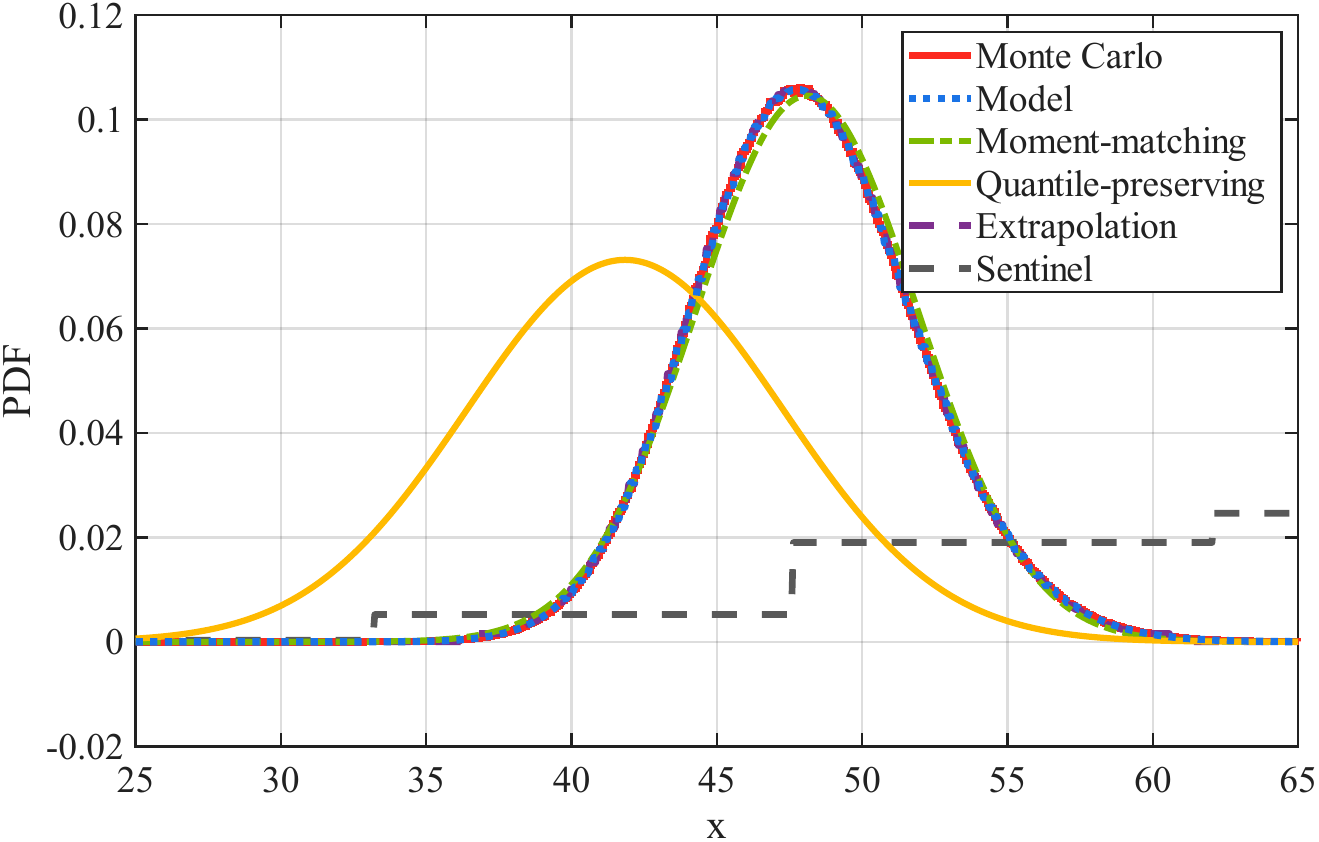}
    \caption{raw PDF}
    \label{fig:bt_pdf_raw}
  \end{subfigure}

  \vspace{0.5mm}

  \begin{subfigure}[t]{0.48\textwidth}
    \centering
    \includegraphics[width=\linewidth]{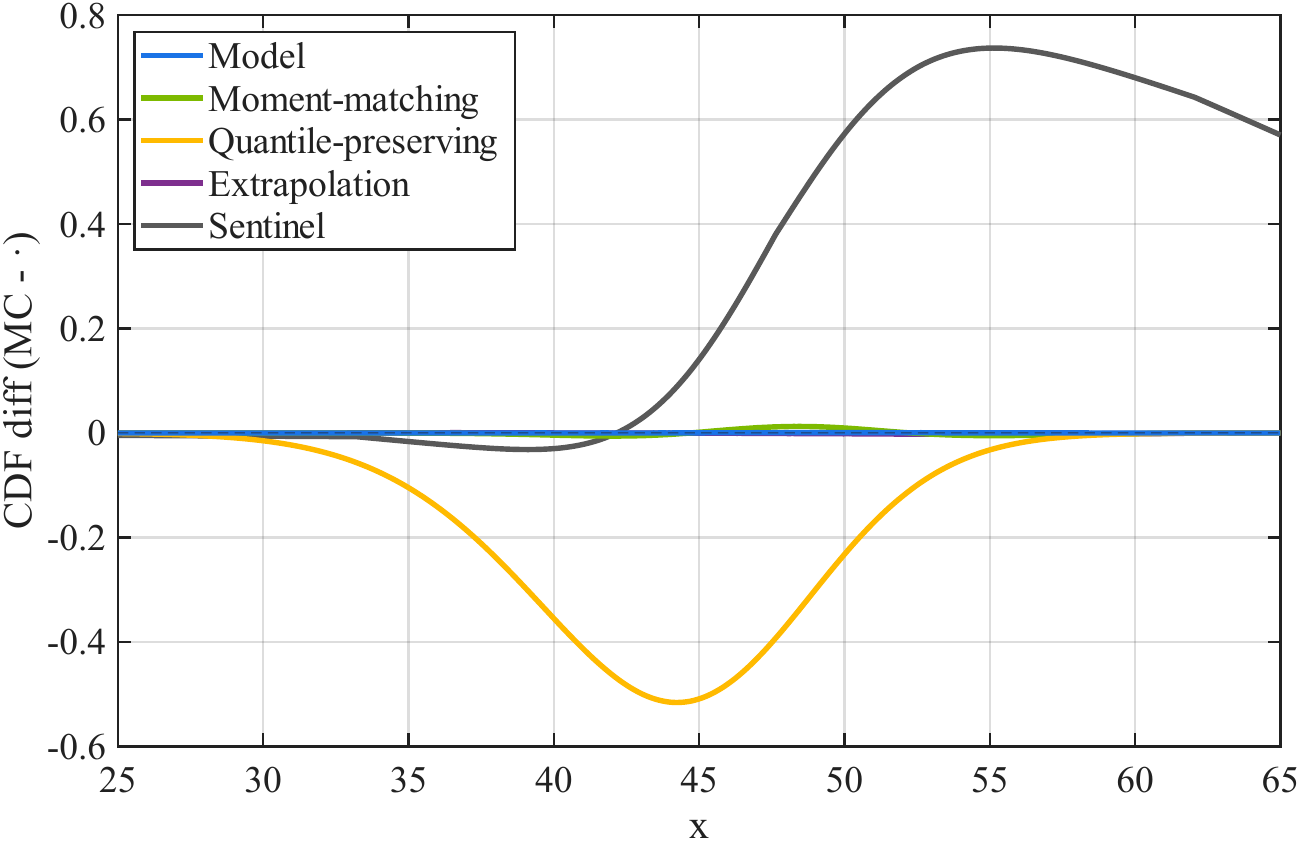}
    \caption{CDF error}
    \label{fig:bt_cdf_diff}
  \end{subfigure}
  \hfill
  \begin{subfigure}[t]{0.48\textwidth}
    \centering
    \includegraphics[width=\linewidth]{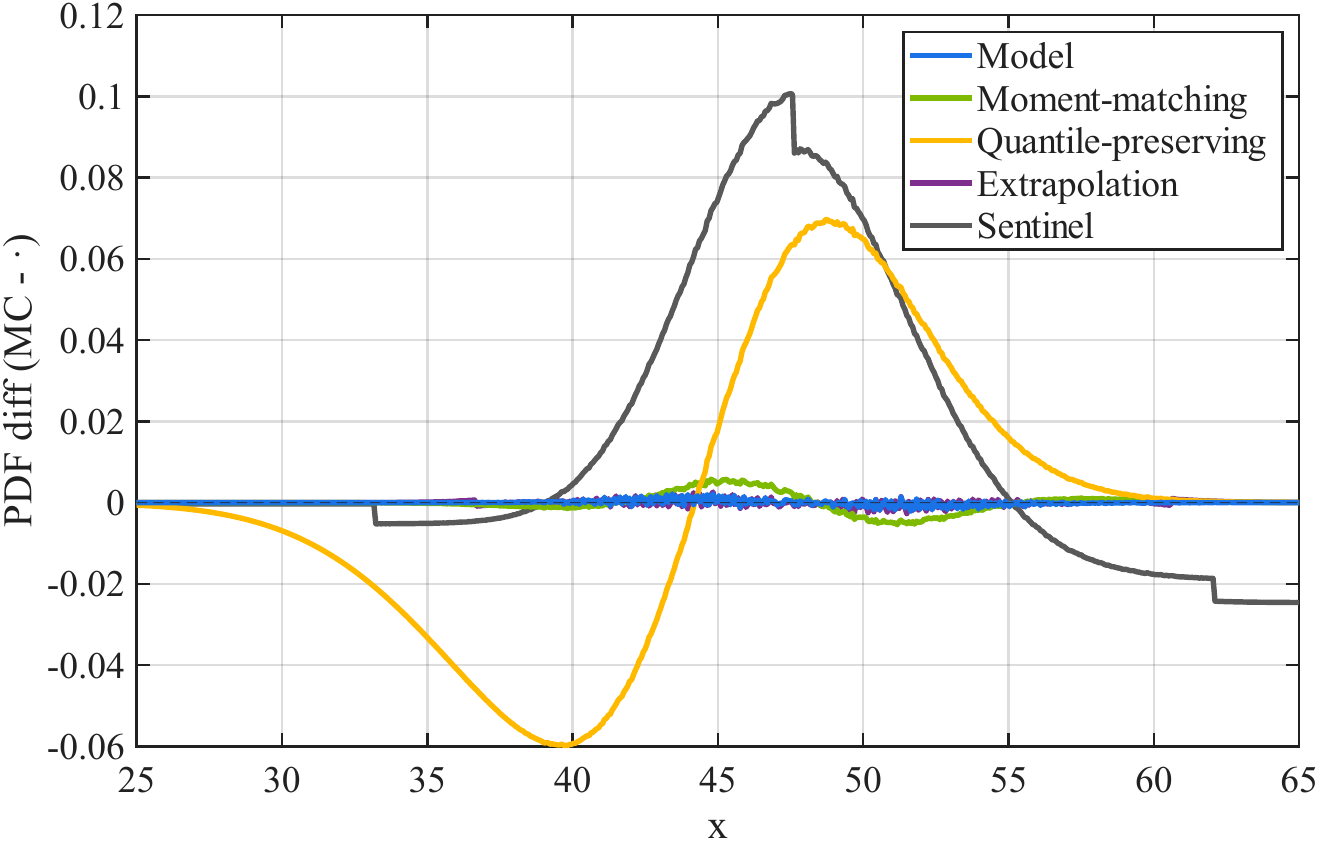}
    \caption{PDF error}
    \label{fig:bt_pdf_diff}
  \end{subfigure}

  \caption{Delay propagation on an inverted binary tree structure: comparison between Monte Carlo simulation, our model results, the Gaussian moment-matching method, the Gaussian quantile-preserving method, the extrapolation method, and the sentinel method.}
  \label{fig:bt_cdf_pdf_compare}
\end{figure}
\begin{table}[ht]
\centering
\caption{Quantile comparison for the inverted binary tree case.}
\label{tab:quantile-bt}
\begin{tabular}{lcccc}
\toprule
\textbf{Method} & $0.135\%$ & $1\%$ & $99\%$ & $99.865\%$ \\
\midrule
Model                & -0.028\% & -0.008\% & 0.009\% & -0.019\% \\
Extrapolation        & 0.228\% & 0.023\% & -0.484\% & -1.489\% \\
Gaussian moment-matching          & -2.066\% & -1.158\% & -0.991\% & -1.798\% \\
Gaussian quantile-preserving & -31.928\% & -26.540\% & -5.240\% & -4.010\% \\
\bottomrule
\end{tabular}
\end{table}

It can be observed from the CDF graph that the Gaussian quantile-preserving method exhibits substantial errors. This is expected, as each max operation within the Gaussian framework would introduce significant model bias, and these errors accumulate during the propagation process. Even when a Gaussian approximation is obtained directly from the Monte Carlo distribution via moment matching, the resulting error remains noticeably larger than that of our model.

Regarding the tail models, the sentinel method fails to produce a reasonable approximation, whereas the extrapolation method performs comparably to our model, though with noticeably larger tail errors. Examination of the PDF error curves further reveals that our model exhibits better approximation behavior, particularly in the tails. This indicates our model's better capability in capturing tail characteristics than linear tail models. 

The high-frequency fluctuations in the PDF error curves are mainly due to the numerical noise introduced by Monte Carlo simulations and the model error we introduced in approximating the intermediate segment PDF as piecewise constants. If higher-degree approximations are employed in the intermediate segment, the magnitude of these oscillations is expected to decrease accordingly.

\subsection{Ladder graph propagation}
To further evaluate the performance of our algorithm on deeper graph structures, we consider a ladder graph with $20$ levels. This setup is analogous to that in~\cite{mishagli_radial_2020}, where this graph was used to assess the Gaussian RBF model. In the Gaussian RBF model, the PDF of each delay distribution is approximated by a mixture of Gaussian basis functions as 
\[
f(x) = \sum\limits_{i = 1}^M \omega_i \exp\left( -\frac{(x-\mu_i)^2}{2\sigma_i^2}\right).
\]

On the $i$-th node, where $i = 1, \dotsc, 20$, we compute the output of this node as
\[
X_i = \max\{X_{i-1} + d_{i-1, i},\; \xi_i\} + d_{i},
\]
where $X_i$ denotes the output delay at node $i$, $d_{i-1,i}$ denotes the delay from node $i-1$ to node $i$, $d_i$ denotes the incremental delay on node $i$, and $\xi_i$ represents the input delay of node $i$, which can be seen as the delay on the edge $(0, i)$. The ladder graph structure is displayed in Fig.~\ref{fig:ladder}.

\begin{figure}[H]
    \centering
    \includegraphics[width=0.7\linewidth]{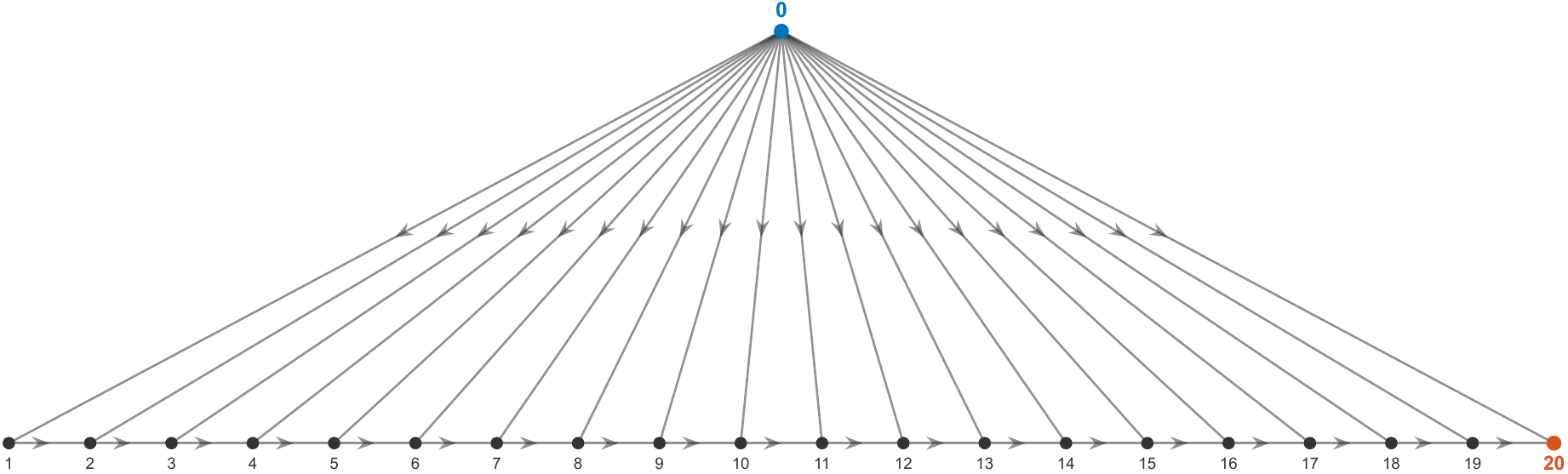}
    \caption{Ladder structure}
    \label{fig:ladder}
\end{figure}

To introduce more variability in this scenario, the input delay on the $i$-th node $\xi_i$ is set as a Gaussian random variable whose mean and standard deviation are drawn from $\mathcal{U}(i, 3i)$ and $\mathcal{U}(1.5\sqrt{i}, 2.5\sqrt{i})$, respectively. Each edge delay of $\{d_{i-1,i}\}$ and each incremental node delay $\{d_i\}$ is also modeled as a Gaussian random variable with mean sampled from $\mathcal{U}(1.0, 3.0)$ and standard deviation sampled from $\mathcal{U}(1.5, 2.5)$. Our primary focus is the output delay at the final node, i.e., $X_{20}$. The ground truth is obtained by a Monte Carlo simulation with $10^7$ samples. 

The resulting PDF and CDF comparisons are presented in Fig.~\ref{fig:ladder_cdf_pdf_compare}, and the corresponding quantile errors are summarized in Table~\ref{tab:quantile-ladder}.
\begin{figure}[t]
  \centering

  \begin{subfigure}[t]{0.48\textwidth}
    \centering
    \includegraphics[width=\linewidth]{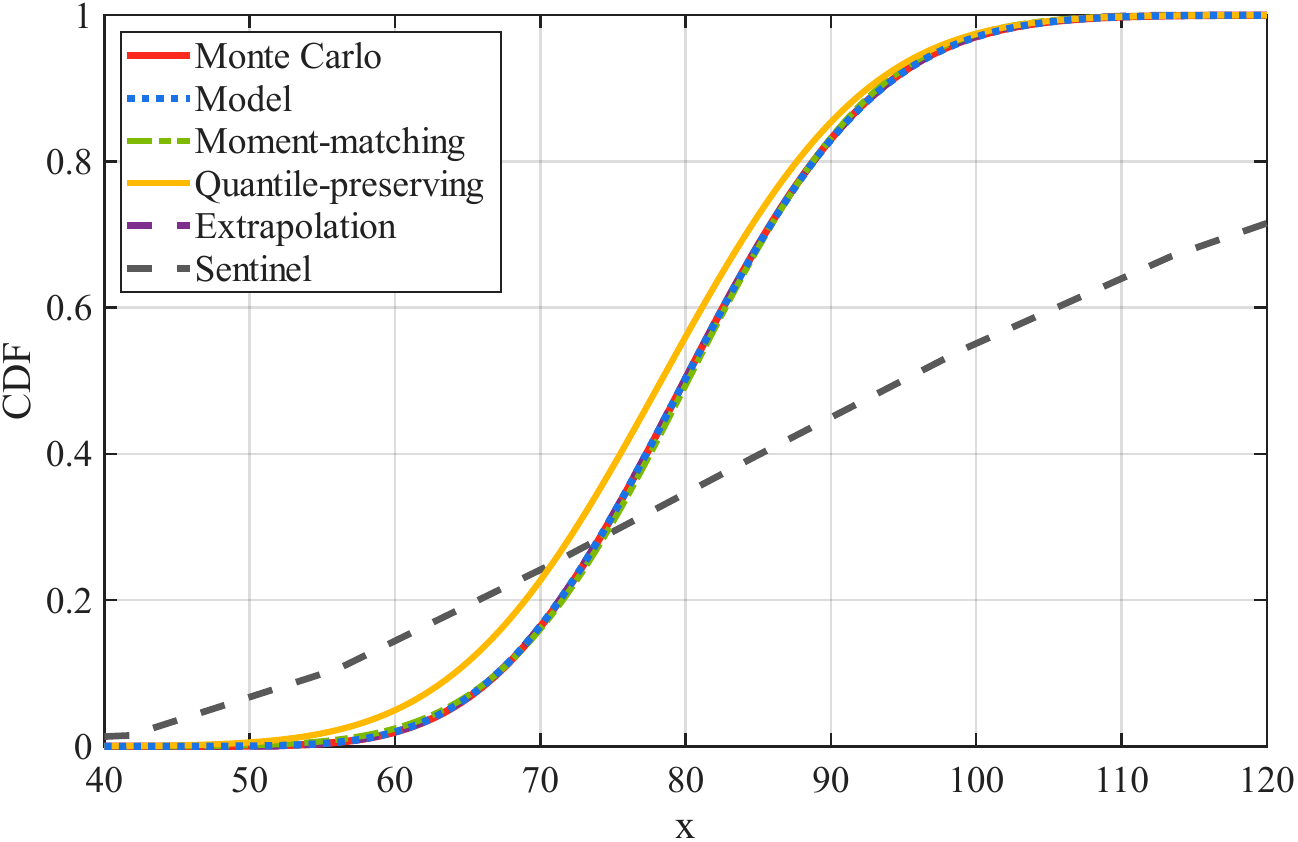}
    \caption{raw CDF}
    \label{fig:ladder_cdf_raw}
  \end{subfigure}
  \hfill
  \begin{subfigure}[t]{0.48\textwidth}
    \centering
    \includegraphics[width=\linewidth]{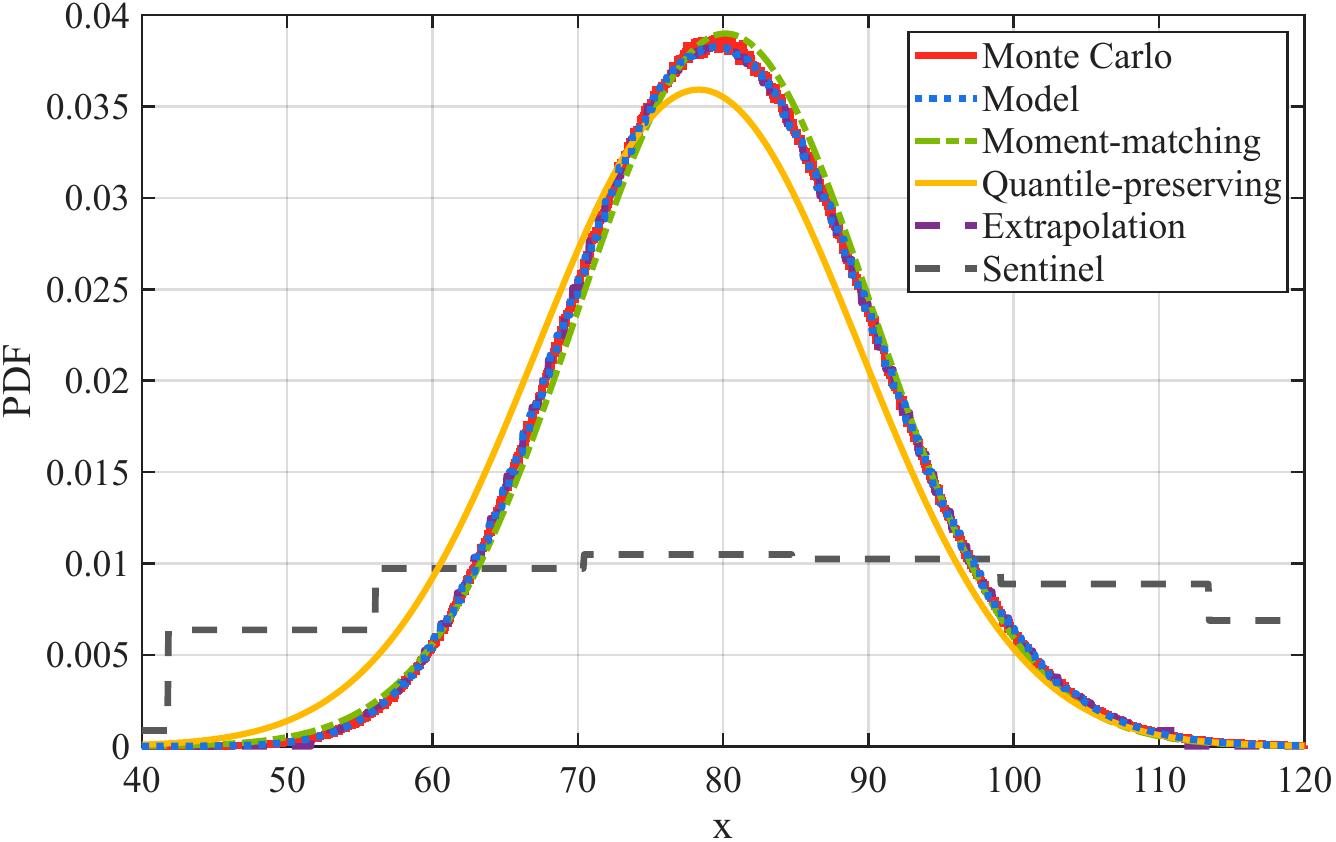}
    \caption{raw PDF}
    \label{fig:ladder_pdf_raw}
  \end{subfigure}

  \vspace{0.5mm}

  \begin{subfigure}[t]{0.48\textwidth}
    \centering
    \includegraphics[width=\linewidth]{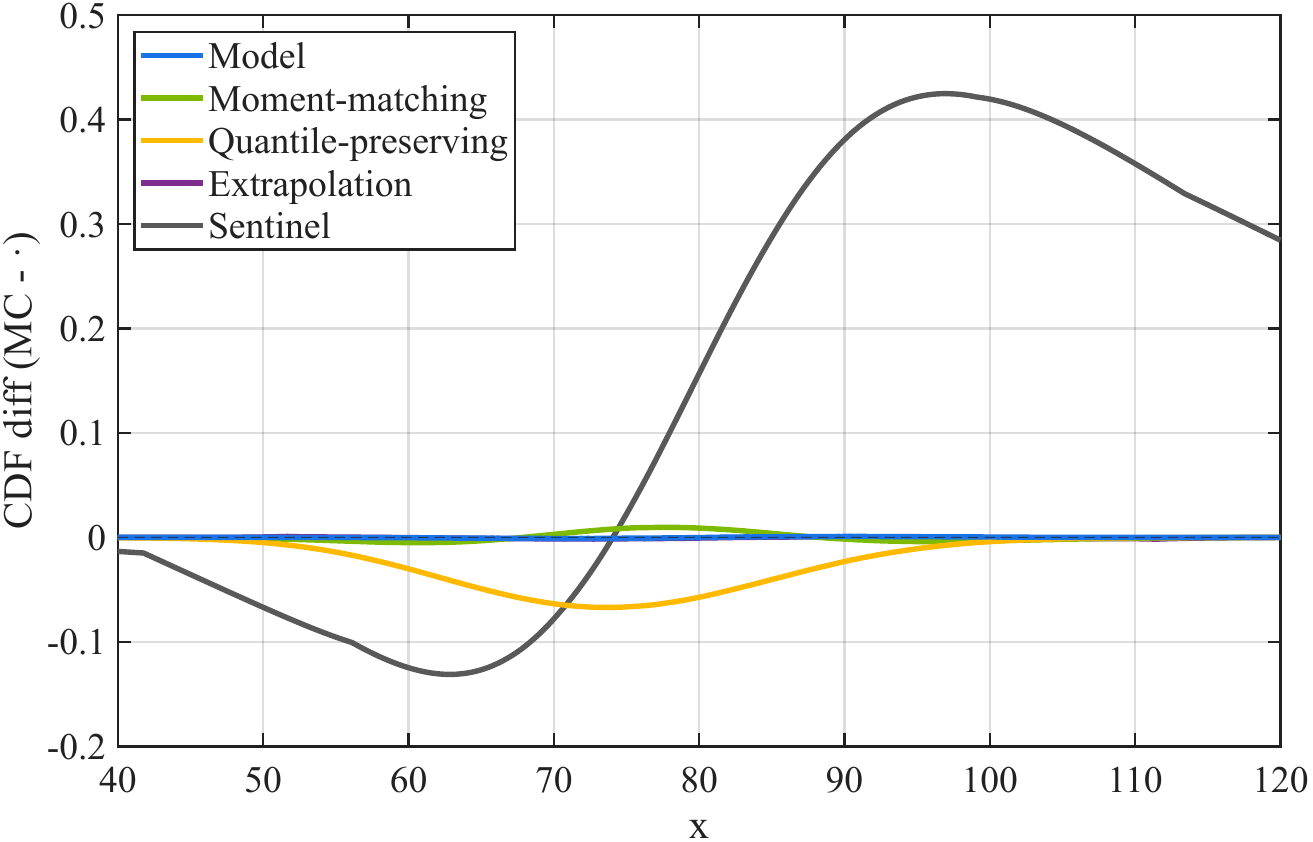}
    \caption{CDF error}
    \label{fig:ladder_cdf_diff}
  \end{subfigure}
  \hfill
  \begin{subfigure}[t]{0.48\textwidth}
    \centering
    \includegraphics[width=\linewidth]{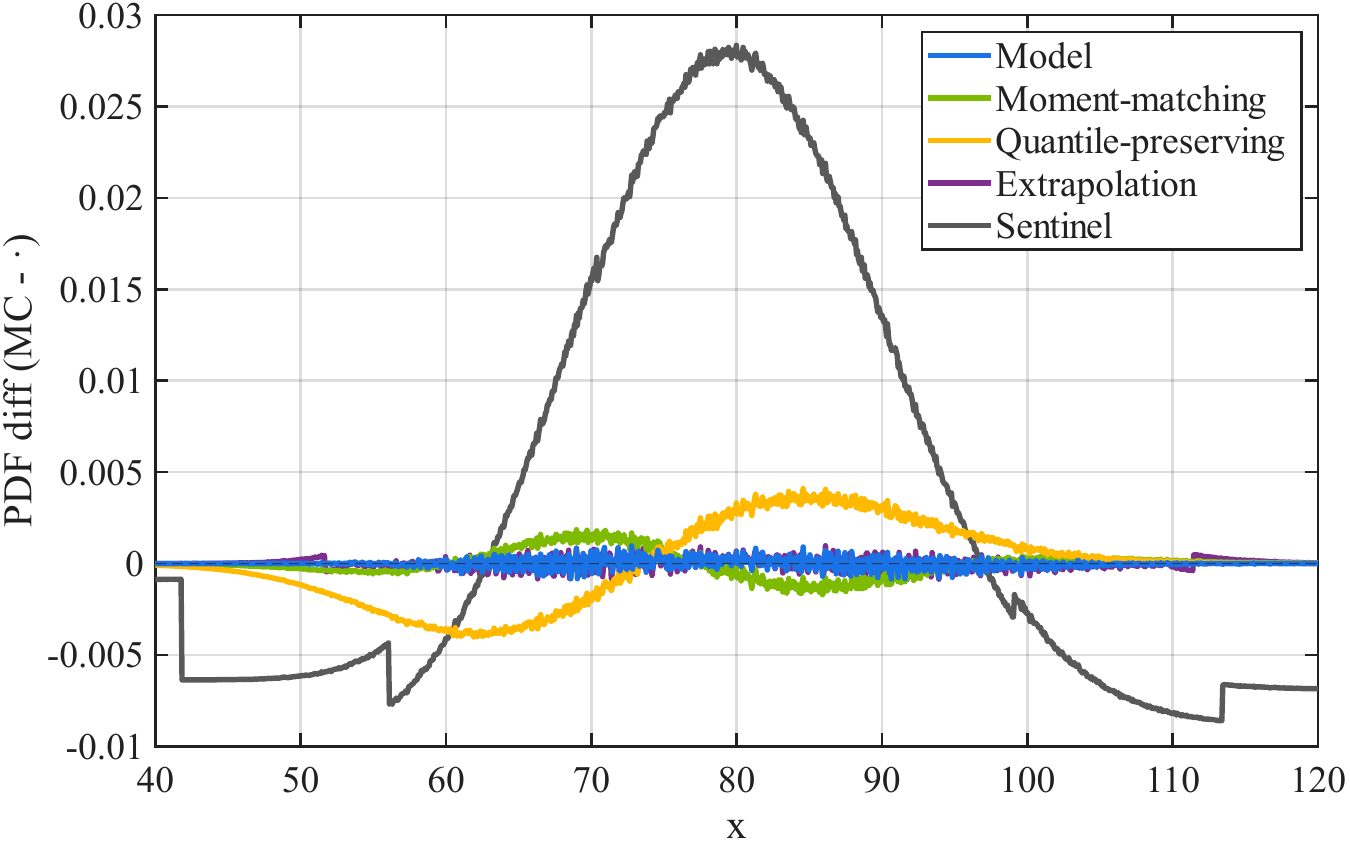}
    \caption{PDF error}
    \label{fig:ladder_pdf_diff}
  \end{subfigure}

  \caption{Delay propagation on a ladder structure: comparison between Monte Carlo simulation, our model results, the Gaussian moment-matching method, the Gaussian quantile-preserving method, the extrapolation method, and the sentinel method.}
  \label{fig:ladder_cdf_pdf_compare}
\end{figure}
\begin{table}[ht]
\centering
\caption{Quantile comparison for the ladder case.}
\label{tab:quantile-ladder}
\begin{tabular}{lcccc}
\toprule
\textbf{Method} & $0.135\%$ & $1\%$ & $99\%$ & $99.865\%$ \\
\midrule
Model                & -0.124\% & -0.112\% & 0.042\% & 0.008\% \\
Extrapolation        & 1.421\% & 0.025\% & -0.321\% & -1.985\% \\
Gaussian moment-matching           & -5.680\% & -2.510\% & -0.764\% & -1.116\% \\
Gaussian quantile-preserving & -14.180\% & -9.202\% & -0.576\% & -0.415\% \\
\bottomrule
\end{tabular}
\end{table}

The results for the ladder graph exhibit trends similar to those observed in the inverted binary tree, though the greater number of propagation stages in the ladder graph leads to additional error accumulation. While our model does not outperform the extrapolation method at the $1\%$ quantile, it remains more robust overall, particularly at the $0.135\%$ and $99.865\%$ quantiles, which are the most critical quantiles.

The CDF error patterns further mirror those in the inverted binary tree experiment. Although the extrapolation method yields comparable CDF curves, its tail errors exhibit pronounced peaks, whereas our model maintains a smoother and more stable error profile. 

From this case and the inverted binary tree case, it can be concluded that our model and algorithm possess strong propagation capability across graph structures, and that our tail model has better approximation power than those reference linear tail models.

\section{Conclusion}\label{sec:conclusion}
In this work, we proposed a three-segment distribution model for stochastic delay propagation on directed acyclic graphs, combining polynomial-Gaussian tail representations with piecewise polynomial approximation in the intermediate segment. We established both $L^2$-convergence and pointwise convergence results, and developed a propagation framework based on CDF sampling and Levenberg--Marquardt tail fitting. Numerical experiments on several graph structures demonstrated the accuracy and robustness of the proposed approach.

Several directions remain for future investigation. Although the numerical computation of the convolution in the sum operation has been optimized, it remains a computational bottleneck. Moreover, our current framework supports only binary operations, so extending it to multi-input settings is a natural next step. Finally, the independence assumption may not hold in practical scenarios; incorporating dependency information, for example through correlation-aware propagation in the spirit of~\cite{Kunhyuk_Statistical_2005}, is another important direction.

\bibliographystyle{unsrt} 
\bibliography{references}

@book{szeg1939orthogonal,
  title={Orthogonal polynomials},
  author={Szeg, Gabor},
  volume={23},
  year={1939},
  publisher={American Mathematical Soc.}
}

@book{cramer1999mathematical,
  title={Mathematical methods of statistics},
  author={Cram{\'e}r, Harald},
  volume={9},
  year={1999},
  publisher={Princeton university press}
}

@article{bosak_statistical_2024,
  title={Statistical static timing analysis via modern optimization lens: I. Histogram-based approach},
  author={Bos{\'a}k, Adam and Mishagli, Dmytro and Mare{\v{c}}ek, Jakub},
  journal={Optimization and Engineering},
  volume={25},
  number={3},
  pages={1405--1429},
  year={2024},
  publisher={Springer}
}

@article{uspensky_development_1926,
  title={On the development of arbitrary functions in series of Hermite's and Laguerre's polynomials},
  author={Uspensky, James V},
  journal={Annals of Mathematics},
  volume={28},
  number={1/4},
  pages={593--619},
  year={1926},
  publisher={JSTOR}
}

@book{chadha2009sta,
  title={Static timing analysis for nanometer designs: A practical approach},
  author={Bhasker, Jayaram and Chadha, Rakesh},
  year={2009},
  publisher={Springer Science \& Business Media}
}

@inproceedings{Kunhyuk_Statistical_2005,
  title={Statistical timing analysis using levelized covariance propagation},
  author={Kang, Kunhyuk and Paul, Bipul C and Roy, Kaushik},
  booktitle={Design, Automation and Test in Europe},
  pages={764--769},
  year={2005},
  organization={IEEE}
}

@inproceedings{Chopra_A_2006,
  title={A new statistical max operation for propagating skewness in statistical timing analysis},
  author={Chopra, Kaviraj and Zhai, Bo and Blaauw, David and Sylvester, Dennis},
  booktitle={Proceedings of the 2006 IEEE/ACM international conference on Computer-aided design},
  pages={237--243},
  year={2006}
}

@article{ben_hcine_fitting_2014,
  title={Fitting the log skew normal to the sum of independent lognormals distribution},
  author={Hcine, Marwane Ben and Bouallegue, Ridha},
  journal={arXiv preprint arXiv:1501.02344},
  year={2015}
}

@article{mishagli_gate_2024,
  title={Gate--Level Statistical Timing Analysis: Exact Solutions, Approximations and Algorithms},
  author={Mishagli, Dmytro and Koskin, Eugene and Blokhina, Elena},
  journal={arXiv preprint arXiv:2401.03588},
  year={2024}
}

@article{ClarkGreatest1961,
  title={The greatest of a finite set of random variables},
  author={Clark, Charles E},
  journal={Operations Research},
  volume={9},
  number={2},
  pages={145--162},
  year={1961},
  publisher={INFORMS}
}

@inproceedings{TsukiyamaCorrelation2001,
  title={A statistical static timing analysis considering correlations between delays},
  author={Tsukiyama, Shuji and Tanaka, Masakazu and Fukui, Masahiro},
  booktitle={Proceedings of the 2001 Asia and South Pacific Design Automation Conference},
  pages={353--358},
  year={2001}
}

@article{FentonLN1960,
  title={The sum of log-normal probability distributions in scatter transmission systems},
  author={Fenton, Lawrence},
  journal={IRE Transactions on communications systems},
  volume={8},
  number={1},
  pages={57--67},
  year={1960},
  publisher={IEEE}
}

@article{JinKurtosis2022,
  title={A statistical cell delay model for estimating the 3$\sigma$ delay by matching kurtosis},
  author={Jin, Leilei and Fu, Wenjie and Yan, Hao and Shi, Longxing},
  journal={IEEE Transactions on Circuits and Systems II: Express Briefs},
  volume={69},
  number={6},
  pages={2932--2936},
  year={2022},
  publisher={IEEE}
}

@inproceedings{berkelaar1997statistical,
  title     = {Statistical Delay Calculation, a Linear Time Method},
  author    = {Berkelaar, M. R. C. M.},
  booktitle = {Proceedings of the 1997 ACM/IEEE International Workshop on Timing Issues in the Specification and Synthesis of Digital Systems},
  pages     = {15--24},
  year      = {1997}
}

@article{ChengPolynomial2008,
  title={Non-Gaussian statistical timing analysis using second-order polynomial fitting},
  author={Lerong Cheng and Jinjun Xiong and Lei He},
  journal={2008 Asia and South Pacific Design Automation Conference},
  year={2008},
  pages={298-303},
  url={https://api.semanticscholar.org/CorpusID:6159807}
}

@article{AzumaApproximating2017,
  title={Approximating the maximum of Gaussians by a Gaussian mixture model for statistical designs},
  author={Daiki Azuma and Shuji Tsukiyama and Masahiro Fukui},
  journal={2017 European Conference on Circuit Theory and Design (ECCTD)},
  year={2017},
  pages={1-4},
  url={https://api.semanticscholar.org/CorpusID:9039226}
}

@article{FreeleyStatistical2018,
  title={Statistical Simulations of Delay Propagation in Large Scale Circuits Using Graph Traversal and Kernel Function Decomposition},
  author={Jennifer Freeley and Dmytro Mishagli and Tom Brazil and Elena Blokhina},
  journal={2018 15th International Conference on Synthesis, Modeling, Analysis and Simulation Methods and Applications to Circuit Design (SMACD)},
  year={2018},
  pages={213-9},
  url={https://api.semanticscholar.org/CorpusID:52020320}
}

@article{mishagli_radial_2020,
  title={Radial Basis Functions Based Algorithms for Non-Gaussian Delay Propagation in Very Large Circuits},
  author={Dmytro Mishagli and Elena Blokhina},
  journal={Computational Science -- ICCS 2020},
  year={2020},
  volume={12141},
  pages={217 - 229},
  url={https://api.semanticscholar.org/CorpusID:219888423}
}

@article{capodaglio_approximation_2021,
  title={Approximation of Probability Density Functions for PDEs with Random Parameters Using Truncated Series Expansions},
  author={Giacomo Capodaglio and Max D. Gunzburger and Henry P. Wynn},
  journal={Vietnam Journal of Mathematics},
  year={2018},
  volume={49},
  pages={685 - 711},
  url={https://api.semanticscholar.org/CorpusID:221865524}
}

@article{DufresneLi2016,
  author = {Daniel Dufresne and HanBo Li},
  title = {Pricing Asian Options: Convergence of Gram--Charlier Series},
  journal = {Actuarial Research Clearing House},
  volume = {2016 Issue 2},
  year = {2016},
}

@article{blinnikov_expansions_1998,
  title={Expansions for nearly Gaussian distributions},
  author={Sergei Blinnikov and Richhild Moessner},
  journal={Astronomy \& Astrophysics Supplement Series},
  year={1997},
  volume={130},
  pages={193-205},
  url={https://api.semanticscholar.org/CorpusID:54041502}
}

@article{Marquardt1963LM,
  title={An Algorithm for Least-Squares Estimation of Nonlinear Parameters},
  author={Donald W. Marquardt},
  journal={Journal of The Society for Industrial and Applied Mathematics},
  year={1963},
  volume={11},
  pages={431-441},
  url={https://api.semanticscholar.org/CorpusID:122360030}
}

@book{gough2009gnu,
  title = {GNU scientific library reference manual},
  author = {Gough, Brian},
  year = {2009},
  publisher = {Network Theory Ltd.},
}

@book{press1989numerical,
  title = {Numerical recipes in Pascal: the art of scientific computing},
  author = {Press, William H},
  volume = {1},
  year = {1989},
  publisher = {Cambridge university press},
}
\appendix
\section*{Appendix}
\section{Right-tail model derivation}\label{AppendixA}
\addcontentsline{toc}{section}{Appendix: Right-Tail Model Derivation}

To complement the left-tail derivation in Sec.~\ref{sec:model} and the tail-fitting algorithm proposed in Sec.~\ref{sec:imp}, the analogous formulation for the right tail is presented here.
Suppose the right-tail PDF is given by Eq.~\eqref{eq:modelPDF} as follows: 
\[
f\bigl(t; \{c_j^{(r)}\}_{j=0}^N, \mu^{(r)}, \sigma^{(r)}\bigr) = \sum_{j=0}^{N} c_j^{(r)} t^j \phi(t;\mu^{(r)}, \sigma^{(r)}), \quad t \ge q_n,
\]
where \( \phi(t;\mu, \sigma) \) is the Gaussian PDF for $\mathcal{N}(\mu,\sigma^2)$.
The corresponding right-tail CDF is computed by
\[
\widetilde{F}\bigl(x; \{c_j^{(r)}\}_{j=0}^N, \mu^{(r)}, \sigma^{(r)}\bigr) = 1 - \int_x^{+\infty} f(t)\dt.
\]
We denote $J_i$ as 
\[
J_i(x;\mu,\sigma) \coloneqq \int_x^{+\infty} t^i \phi(t; \mu, \sigma)\, \dt.
\]
We derive the recurrence relations for \( J_i(x;\mu,\sigma) \) inductively:
\begin{align*}
J_0(x;\mu,\sigma) &= 1 - \Phi(x;\mu,\sigma), \\
J_1(x;\mu,\sigma) &= \mu J_0(x;\mu,\sigma) + \sigma \phi(x;\mu,\sigma), \\
J_i(x;\mu,\sigma) &= \mu J_{i-1}(x;\mu,\sigma) + (i-1)\sigma^2 J_{i-2}(x;\mu,\sigma) + x^{i-1}\sigma \phi(x;\mu,\sigma), \quad i \ge 2.
\end{align*}
Thus, for \( x \ge q_n \), the right-tail CDF lies in the span of
\[
\text{span}\left\{ 1 - \Phi\bigl(x;\mu^{(r)}, \sigma^{(r)}\bigr),\, \phi\bigl(x;\mu^{(r)}, \sigma^{(r)}\bigr),\, x\phi\bigl(x;\mu^{(r)}, \sigma^{(r)}\bigr),\, \dots,\, x^{N-1}\phi\bigl(x;\mu^{(r)}, \sigma^{(r)}\bigr) \right\}.
\]
For the max operator, we focus on the optimization of the following two expressions:
\[
\left\{
\begin{array}{ll}
\displaystyle
\int_{q_n}^{+\infty}
\Bigl(
1 - F_1(x)F_2(x)
- \sum_{j=0}^{N} c_j^{(r)} J_{j}\bigl(x;\mu^{(r)},\sigma^{(r)}\bigr)
\Bigr)^{2} \dx, 
\\[2ex]
\displaystyle
1 - F_1(x)F_2(x)
- \sum_{j=0}^{N} c_j^{(r)} J_{j}\bigl(x;\mu^{(r)},\sigma^{(r)}\bigr),
\quad \forall x \ge q_n.
\end{array}
\right.
\]

The tail-fitting optimization problem for the right tail in our algorithm is analogous to that for the left tail. Suppose $y_i$ is the exact CDF at $x_i$. We solve the following nonlinear least-squares problem:
\begin{equation*}
\min \sum_{i=1}^k \biggl[ \Bigl( \widetilde{F}\bigl(x_i; \{c_j^{(r)}\}_{j=0}^N, \mu^{(r)}, \sigma^{(r)}\bigr) - y_i \Bigr) \Big/ (1-y_i) \biggr]^2 + \lambda \bigl(c_0^{(r)}-1\bigr)^2 + \lambda\biggl( \sum\limits_{j = 1}^N \bigl(c_j^{(r)}\bigr)^2\biggr),
\end{equation*}

The Jacobian matrix for this expression can be computed by differentiating each \( J_j(x_i; \mu^{(r)}, \sigma^{(r)}) \) for the parameters \( \{c_j^{(r)}\}_{j=0}^N \), \( \mu^{(r)} \), and \( \sigma^{(r)} \). We use the same reparameterization \( \sigma^{(r)} = \exp(p_\sigma) \) to ensure positivity.

\section{Model sensitivity to correlation}\label{AppendixB}
While the proposed model and propagation algorithm demonstrate strong expressive power and robust performance under the independence assumption, their accuracy may degrade in settings where correlations within the graph are non-negligible. To examine this limitation, we consider the simple graph shown in Fig.~\ref{fig:combCase}.
\begin{figure}[H]
    \centering
    \includegraphics[width=0.4\linewidth]{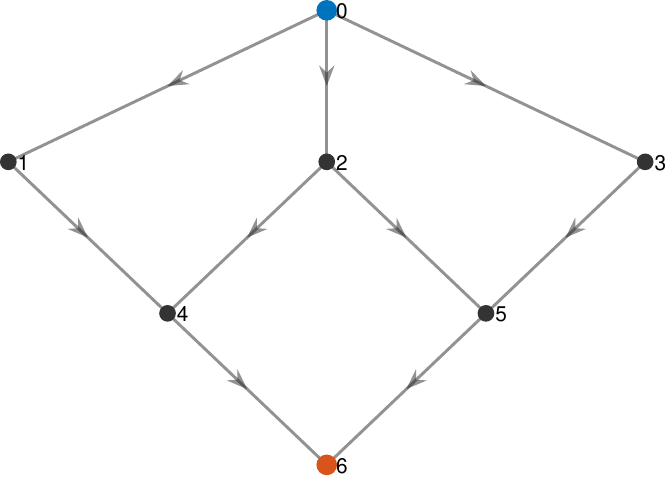}
    \caption{Diamond tree structure}
    \label{fig:combCase}
\end{figure}

The incremental delays on each node and edge are modeled as independent Gaussian random variables, whose $\mu$ and $\sigma$ are drawn randomly from $\mathcal{U}(2.0,5.0)$ and $\mathcal{U}(1.5, 2.0)$ respectively. 

As illustrated in Fig.~\ref{fig:combCase}, the four nodes $\{2,4,5,6\}$ form a diamond structure, which induces statistical dependence when the input delay of node~$6$ is computed via a max operation over two paths that share common ancestors.

\begin{figure}[H]
    \centering
    \begin{subfigure}[t]{0.48\textwidth}
        \centering
        \includegraphics[width=\linewidth,height=5.0cm]{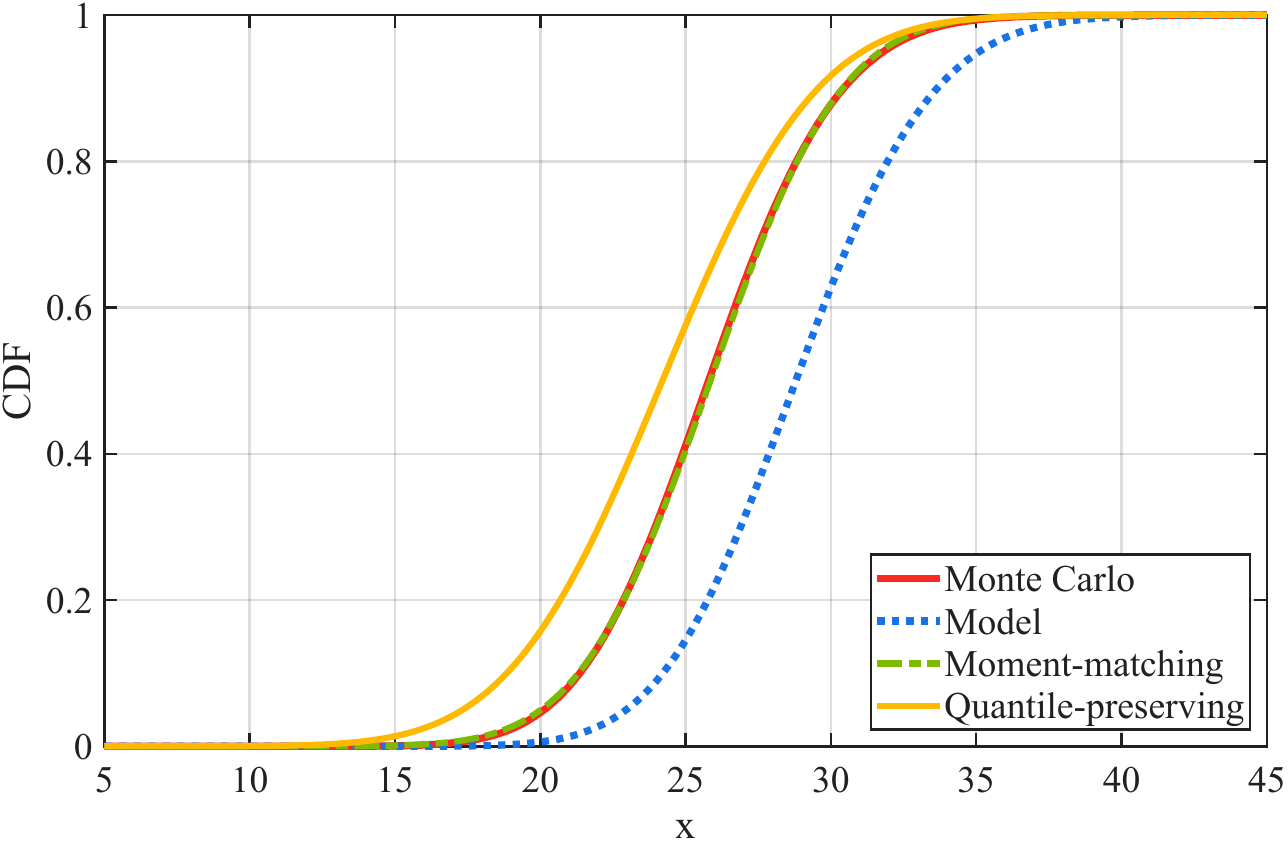}
        \caption{CDF comparison}
        \label{fig:corr_cdf}
    \end{subfigure}
    \hfill
    \begin{subfigure}[t]{0.48\textwidth}
        \centering
        \includegraphics[width=\linewidth,height=5.0cm]{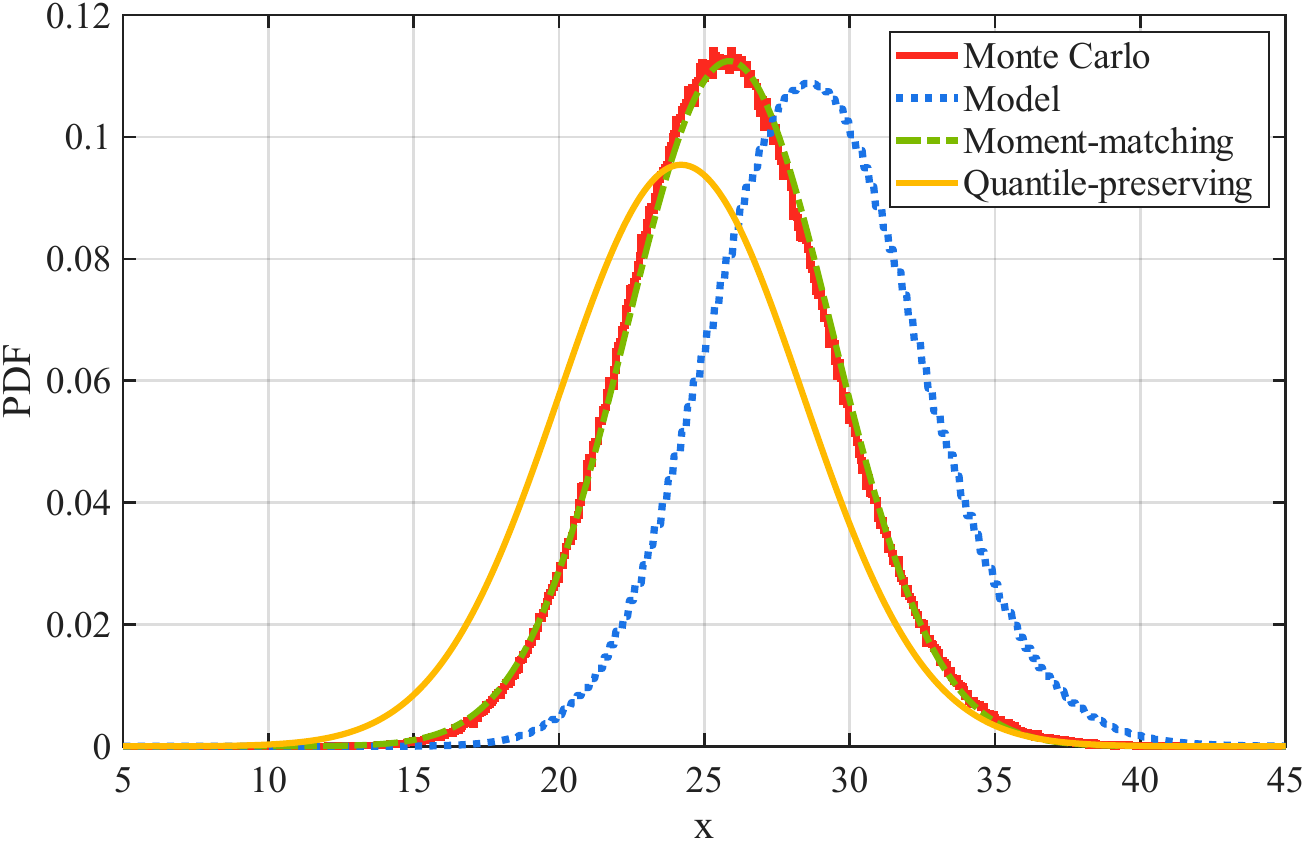}
        \caption{PDF comparison}
        \label{fig:corr_pdf}
    \end{subfigure}

    \caption{Comparison of \subref{fig:corr_cdf} CDF and \subref{fig:corr_pdf} PDF of the delay propagation on a diamond tree structure between Monte Carlo simulation, our model results, the moment-matching method, and the Gaussian quantile-preserving method.}
    \label{fig:corr_cdf_pdf_compare}
\end{figure}

As expected, Fig.~\ref{fig:corr_cdf_pdf_compare} shows that our model exhibits noticeable deviations from the Monte Carlo ground truth. This behavior is a characteristic of propagation schemes that ignore correlations among random variables.

\end{document}